\documentclass{amsart}

\usepackage{graphicx} 
\usepackage{amsmath}
\usepackage{amsthm}
\usepackage{amsfonts}
\usepackage{mathrsfs}
\usepackage{dsfont}
\usepackage{color}
\usepackage{enumerate}
\usepackage{tikz-cd}
\usepackage[margin=42mm]{geometry}

\newcommand{\reals}{\mathbb{R}}

\newcommand{\complex}{\mathbb{C}}

\newcommand{\paraa}[1]{\big(#1\big)}

\newcommand{\sgn}{\operatorname{sgn}}

\newcommand{\Hom}{\operatorname{Hom}}

\newcommand{\spacearound}[1]{\quad#1\quad}
\newcommand{\equivalent}{\spacearound{\Leftrightarrow}}
\renewcommand{\implies}{\spacearound{\Rightarrow}}

\newtheorem{theorem}{Theorem}[section]
\newtheorem{corollary}[theorem]{Corollary}
\newtheorem{lemma}[theorem]{Lemma}
\newtheorem{proposition}[theorem]{Proposition}
\newtheorem{example}[theorem]{Example}
\theoremstyle{definition}
\newtheorem{definition}[theorem]{Definition}
\theoremstyle{remark}
\newtheorem{remark}[theorem]{Remark}
\numberwithin{equation}{section}

\newcommand{\A}{\mathcal{A}}

\newcommand{\B}{\mathcal{B}}
\renewcommand{\mid}{\mathds{1}}

\renewcommand{\d}{\partial}

\newcommand{\Der}{\operatorname{Der}}
\newcommand{\qand}{\quad\text{and}\quad}
\newcommand{\qqand}{\qquad\text{and}\qquad}

\newcommand{\g}{\mathfrak{g}}

\newcommand{\Omegag}[1]{\Omega^{#1}_{\g}}
\newcommand{\Omegabg}[1]{\bar{\Omega}^{#1}_{\g}}
\newcommand{\Omegaonebg}{\Omegabg{1}}
\newcommand{\Omegaonebgd}{(\Omegabg{1})^\ast}
\newcommand{\Omegaoneg}{\Omega^1_{\g}}
\newcommand{\Omegaonegh}{\widehat{\Omega}^1_{\g}}

\renewcommand{\emph}[1]{\textit{#1}}

\newcommand{\hh}{\hat{h}}
\newcommand{\hhi}{\hh^{-1}}
\newcommand{\hi}{h^{-1}}

\newcommand{\vh}{\varphi_h}

\newcommand{\nablat}{\tilde{\nabla}}

\newcommand{\id}{\operatorname{id}}
\newcommand{\Mat}{\operatorname{Mat}}
\newcommand{\Cg}{\mathscr{C}_{\g}}
\newcommand{\CgT}{\Cg^{\mathrm{T}}}
\newcommand{\CgLC}{\Cg^{\mathrm{LC}}}
\newcommand{\CgReg}{\Cg^\mathrm{Reg}}
\newcommand{\ZA}{Z(\A)}
\newcommand{\ZhA}{\hat{Z}(\A)}

\newcommand{\HomReg}{\Hom^{\mathrm{Reg}}}
\newcommand{\Mh}{\widehat{M}}
\newcommand{\thalf}{\tfrac{1}{2}}
\newcommand{\thetat}{\tilde{\theta}}
\newcommand{\X}{\mathcal{X}}

\title[]{On the existence of noncommutative Levi-Civita connections in derivation based calculi}
\date{}

\author{Joakim Arnlind}
\address[Joakim Arnlind]{Dept. of Math.\\
Link\"oping University\\
581 83 Link\"oping\\
Sweden}
\email{joakim.arnlind@liu.se}

\author{Victor Hildebrandsson}
\address[Victor Hildebrandsson]{Dept. of Math.\\
Link\"oping University\\
581 83 Link\"oping\\
Sweden}
\email{victor.hildebrandsson@liu.se}

\begin{document}

\begin{abstract}
    We study the existence of Levi-Civita connections, i.e torsion free connections compatible with a hermitian form, in the setting of derivation based noncommutative differential calculi over $\ast$-algebras. We prove a necessary and sufficient condition for the existence of Levi-Civita connections in terms of the image of an operator derived from the hermitian form. Moreover, we identify a necessary symmetry condition on the hermitian form that extends the classical notion of metric symmetry in Riemannian geometry. The theory is illustrated with explicit computations for free modules of rank three, including noncommutative 3-tori. We note that our approach is algebraic and does not rely on analytic tools such as $C^\ast$-algebra norms. 
\end{abstract}

\maketitle

\tableofcontents

\section{Introduction}

\noindent
Noncommutative geometry aims to understand geometric concepts in a noncommutative context. The transition from a commutative algebra of functions on a manifold to a noncommutative algebra representing an abstract "space" requires an algebraic formulation of geometry that allows for generalizations to a noncommutative context. For instance, the fact that the category of (locally compact Hausdorff) topological spaces is equivalent to the category of commutative $C^\ast$-algebras, or that the category of vector bundles is equivalent to the category of finitely generated projective modules, 
allow for a conceptual definition of the corresponding noncommutative objects. Having reasonable definitions at hand, it is a natural question to ask if classical theorems in geometry have noncommutative counterparts? Does one encounter noncommutative effects which opens up new questions to answer?

In particular, we are interested in metric, or Riemannian, aspects of the theory. A Riemannian manifold $(M,g)$ is a differentiable manifold $M$ together with a Riemannian metric $g$, i.e. an inner product on the tangent space at each point of $M$. In other words, if $\X(M)$ denotes the set of smooth vector fields on $M$ then
\begin{align*}
    g:\X(M)\times\X(M)\to C^\infty(M)   
\end{align*}
is a $C^\infty(M)$-bilinear nondegenerate symmetric form on $\X(M)$. A fundamental result in Riemannian geometry states that there exists a unique torsion free and metric connection $\nabla$ on $\X(M)$, i.e. a connection satisfying
\begin{align*}
    &\nabla_X Y-\nabla_Y X - [X,Y] = 0\\
    &X\paraa{g(Y,Z)} = g(\nabla_X Y,Z) + g(Y,\nabla_X Z)
\end{align*}
for $X,Y,Z\in\X(M)$. Equivalently, one may formulate the above in terms of a connection on the set of differential 1-forms $\Omega^1(M)$ as
\begin{align*}
    &(\nabla_{X}\omega)(Y) - (\nabla_{Y}\omega)(X)-d\omega(X,Y) = 0\\
    &X\paraa{g^{-1}(\omega,\eta)} = g^{-1}(\nabla_{X}\omega,\eta)+g^{-1}(\omega,\nabla_X\eta)
\end{align*}
for $X,Y\in\X(M)$ and $\omega,\eta\in\Omega^1(M)$, where $g^{-1}$ denotes the inverse metric on $\Omega^1(M)$, which turns out to be a more suitable formulation for noncommutative geometry.
The main goal of this paper is to obtain a better understanding of noncommutative Levi-Civita connections. 

Over the last decade, there has been quite some progress in the understanding of Riemannian aspects of noncommutative geometry. The question of noncommutative Levi-Civita connections can be asked in several different contexts; e.g. for quantum groups \cite{bm:starCompatibleConnections,mw:quantum.koszul,bm:Quantum.Riemannian.geometry,aw:metric.comp.lc.qg,a:lc.braided,ail:lc.quantum.spheres}, differential graded algebras and spectral triples \cite{bgl:lc.vector.fields,MR4074414,mr:existence.uniqueness.LC} and derivation based calculi \cite{r:leviCivita,aw:cgb.sphere,aw:curvature.three.sphere,a:levi-civita.class.nms,a:lc.kronecker} (cf. \cite{fmr:comparison.lc} for a recent comparison of different approaches, where the existence and uniqueness of Levi-Civita connections is studied in the context of centered bimodules). In general, these results give conditions for the existence of Levi-Civita connections for the corresponding concepts of connections, torsion and metric compatibility, depending on the particular way the noncommutative Riemannian geometry is realized. 

In this paper we would like to return to a so called derivation based approach to noncommutative geometry pioneered by M. Dubois-Violette \cite{dv:calculDifferentiel}, where a differential calculus over a $\ast$-algebra is defined by a choice of derivations. In this context there are natural and simple concepts of torsion and metric compatibility but, to the best of our knowledge, no general results concerning the existence of Levi-Civita connections. We emphasize that our approach and results are purely algebraic in nature and does not depend on any analytic structure of the algebra (e.g. the existence of a norm); in fact, we have previously considered examples that cannot be embedded into $C^\ast$-algebras due to the existence of hermitian elements that square to zero (cf. \cite{a:lc.kronecker}). We believe that this approach complements more analytic results (e.g. as found in
\cite{mr:existence.uniqueness.LC}).

The main objects of study in this paper are derivation based calculi, which consist of a unital $\ast$-algebra $\A$ together with a preferred choice of derivations $\g\subseteq\Der(\A)$. This data generates a differential graded algebra $\Omega_{\g}$ and the analogue of (hermitian) metrics is given by (invertible) hermitian forms on the set $\Omegaoneg$ of "noncommutative 1-forms". A noticable difference, compared to hermitian metrics arising on a Riemannian manifold, is a less degree of symmetry; by definition, a hermitian form satsifies
\begin{align*}
    h(\omega,\eta)^\ast = h(\eta,\omega)
\end{align*}
for $\omega,\eta\in\Omegaoneg$; however, if $(M,g)$ is a Riemanian manifold and $h(\omega,\eta)=g^{-1}(\bar{\omega},\eta)$ then 
\begin{align}\label{eq:extra.sym.h}
    h(\bar{\omega},\eta) = g^{-1}(\omega,\eta)=g^{-1}(\eta,\omega)
    =h(\bar{\eta},\omega)
\end{align}
due to the symmetry of the Riemannian metric $g$. Such a symmetry is in general not satisfied for a hermitian form on a module over a noncommutative algebra, and imposing it gives a rather restrictive condition. However, we will establish a weaker notion of symmetry (applied to the inverse of the hermitian form) which will be necessary for the existence of Levi-Civita connections, serving as a noncommutative replacement for the classical symmetry of the Riemannian metric.

The paper is organized as follows: After having introduced the basic concepts of derivation based calculi in Sections \ref{sec:nc.diff.forms} and \ref{sec:derivations.based.calculi}, we proceed to characterize the sets of torsion free and metric compatible connections in Sections \ref{sec:conn.comp.herm} and \ref{sec:torsionfree.conn}, respectively, as well as provide explicit constructions of all such connections. 
In Section \ref{sec:lc.connections} we combine these results to describe the set of Levi-Civita connections as the inverse image of a certain operator and we prove a necessary and sufficient condition for the existence of Levi-Civita connections. Consequently, we find a necessary condition for the hermitian form, which we interpret as a weak symmetry condition related to \eqref{eq:extra.sym.h}. In the case of free modules, this condition turns out to also be sufficient. Finally, in Section \ref{sec:lc.free.modules} we illustrate these concepts for free modules of rank $3$ (and, in particular, for noncommutative 3-tori) where Levi-Civita connections are explicitly computed for a class of hermitian forms.  

\section{Noncommutative differential forms}\label{sec:nc.diff.forms}

\noindent
Let us start by briefly recalling the construction of noncommutative differential forms in a derivation based calculus. To this end, let $\A$ be a unital associative $\ast$-algebra
over $\complex$, let $\ZA$ denote the center of $\A$ and let $\Der(\A)$ denote the set of derivations of $\A$. We shall consider $\Der(\A)$ to be a module over the commutative ring $Z(\A)$ in the standard way; i.e. $(z\cdot \d)(a)=z\d(a)$
for $z\in Z(\A)$, $a\in\A$ and $\d\in\Der(\A)$. Moreover, $\Der(\A)$ can be equipped with a $\ast$-structure $\ast:\Der(\A)\to\Der(\A)$ induced from $\A$ as
\begin{align*}
    \d^\ast(a) = \d(a^\ast)^\ast
\end{align*}
for $\d\in\Der(\A)$ and $a\in\A$, and we say that a subspace $U\subseteq\Der(\A)$ is $\ast$-closed if $\d\in U$ implies that $\d^\ast\in U$.

In the following, we will consider $\ast$-closed complex sub-Lie algebras $\g\subseteq\Der(\A)$, which are not necessarily $\ZA$-submodules of $\Der(\A)$. In this setting, we will consider additive maps $\omega:\g\to\A$ that are partially linear over $\ZA$ in the sense that
\begin{align*}
    \omega(z\d) = z\omega(\d)
\end{align*}
for $z\in\ZA$ and $\d\in\g$ whenever $z\d\in\g$. The set of such additive partially linear maps over $\ZA$ will be denoted by $\Hom_{\ZhA}(\g,\A)$. Since $\A$ is assumed to be a unital algebra and $\g$ a complex Lie algebra, it follows that $\Hom_{\ZhA}(\g,\A)\subseteq\Hom_{\complex}(\g,\A)$, with equality if the center of $\A$ is trivial. On the other hand, if $\g$ is a $\ZA$-module then $\Hom_{\ZhA}(\g,\A)=\Hom_{\ZA}(\g,\A)$.

One introduces noncommutative differential $k$-forms as follows (cf. \cite{dv:calculDifferentiel}).

\begin{definition}
    Let $\A$ be a unital $\ast$-algebra over $\complex$ and let $\g\subseteq\Der(\A)$ be a $\ast$-closed complex Lie algebra. Define $\Omegabg{k}$ be the set of alternating partially $Z(\A)$-multilinear maps $\omega:\g^k\to\A$, i.e.
    \begin{align*}
        &\omega(\d_1,\ldots,\d_k) = \sgn(\sigma)\omega(\d_{\sigma(1)},\ldots,\d_{\sigma(k)})\\
        &\omega(z\d,\d_2,\ldots,\d_k) = z\omega(\d,\d_2,\ldots,\d_k)
    \end{align*}
    for $\sigma\in S_k$, $z\in\ZA$ and $\{\d,\d_1,\ldots,\d_k\}\subseteq\g$ such that $z\d\in\g$ (where $S_k$ denotes the symmetric group of order $k$). Furthermore, we set
    \begin{align*}
        \Omegabg{} = \bigoplus_{k\geq 0}\Omegabg{k},      
    \end{align*}
    with $\Omegabg{0}=\A$.
\end{definition}

\noindent
Note that the dimension of the Lie algebra $\g$ bounds the maximum order of nonzero differential forms, due to the antisymmetry property. Namely, if $\dim(\g)=n$ then $\Omegabg{k}=0$ for $k>n$.
For $\omega\in\Omegabg{k}$ and
$\tau\in\Omegabg{l}$ one defines $\omega\tau\in\Omegabg{k+l}$ as
\begin{align*}
  (\omega\tau)(\d_1,\ldots,\d_{k+l}) =
  \frac{1}{k!l!}\sum_{\sigma\in S_{k+l}}\sgn(\sigma)\omega(\d_{\sigma(1)},\ldots,\d_{\sigma(k)})
  \tau(\d_{\sigma(k+1)},\ldots,\d_{\sigma(k+l)}),
\end{align*}
and introduces
$d_k:\Omegabg{k}\to\Omegabg{k+1}$ as $d_0a(\d_0)=\d_0(a)$ for
$a\in\Omegabg{0}=\A$ and
\begin{align*}
  d_k\omega(\d_0,\ldots,\d_{k}) =
  &\sum_{i=0}^k(-1)^i\d_i\paraa{\omega(\d_0,\ldots,\hat{\d}_{i},\ldots,\d_k)}\\
  &\qquad+\sum_{0\leq i<j\leq k}(-1)^{i+j}\omega\paraa{[\d_i,\d_j],\d_0,\ldots,\hat{\d}_i,\ldots,\hat{\d}_j,\ldots,\d_k},
\end{align*}
for $\omega\in\Omegabg{k}$ with $k\geq 1$, satisfying $d_{k+1}d_{k}=0$, where $\hat{\d}_i$ denotes the omission
of $\d_i$ in the argument. When there is no risk of confusion, we
shall omit the index $k$ and simply write
$d:\Omegabg{k}\to\Omegabg{k+1}$. For instance, if
$\omega\in\Omegabg{1}$ then
\begin{align*}
  d\omega(\d_0,\d_1) = \d_0\omega(\d_1)-\d_1\omega(\d_0)-\omega([\d_0,\d_1])
\end{align*}
and, in general, we say that $\omega\in\Omegabg{k}$ is closed if $d\omega=0$. 
Moreover, we note that the graded product rule is satisfied
\begin{align*}
  d(\omega\eta) = (d\omega)\eta + (-1)^k\omega d\eta 
\end{align*}
for $\omega\in\Omegabg{k}$ and $\eta\in\Omegabg{l}$. The above definitions make $\Omegabg{}$ into a differential graded algebra with $\deg(\omega)=k$ if $\omega\in\Omegabg{k}$.
The $\A$-bimodule structure of $\Omegabg{k}$ is given by
\begin{align*}
  &(a\omega)(\d_1,\ldots,\d_k) = a\omega(\d_1,\ldots,\d_k)\\
  &(\omega a)(\d_1,\ldots,\d_k) = \omega(\d_1,\ldots,\d_k)a
\end{align*}
for $a\in\A$, $\omega\in\Omegabg{k}$, and a $\ast$-structure may be introduced as
\begin{align*}
  \omega^\ast(\d_1,\ldots,\d_k) = \omega(\d_1^\ast,\ldots,\d_k^\ast)^\ast
\end{align*}
satisfying $(a\omega b)^\ast = b^\ast\omega^\ast a^\ast$, making
$\Omegabg{}$ into a $\ast$-bimodule with
\begin{align*}
  (\omega\eta)^\ast = (-1)^{kl}\eta^\ast\omega^\ast\qqand
  d(\omega^\ast) = (d\omega)^\ast
\end{align*}
for $\omega\in\Omegabg{k}$ and $\eta\in\Omegabg{l}$. 

We will consider the dual $\Omegaonebgd$ as the dual of $\Omegaonebg$ as a left module. That is, $\Omegaonebgd$ is a right module with
\begin{align*}
    (\phi a)(\omega) = \phi(\omega)a
\end{align*}
for $\phi\in\Omegaonebgd$, $\omega\in\Omegaoneg$ and $a\in\A$. 
One embeds $\g$ into $\Omegaonebgd$ through $\varphi:\g\to(\Omegaonebg)^\ast$, defined as
\begin{align*}\label{eq:def.varphi.into.dual}
    \varphi(\d)(\omega) = \omega(\d)
\end{align*}
for $\omega\in\Omegaonebg$;  it follows that $\varphi(z\d)=z\varphi(\d)$ for $\d\in\g$ and $z\in\ZA$ such that $z\d\in\g$. Note that $\varphi$ is indeed injective since
\begin{align*}
    \varphi(\d) = 0\implies
    \varphi(\d)(da)=0\,\,\forall a\in\A\implies
    \d a=0 \,\,\forall a\in\A \implies \d=0.
\end{align*}
In the following, we shall primarily be interested in a subalgebra $\Omegag{}\subseteq\Omegabg{}$, generated by the algebra $\A$; that is, one sets $\Omegag{0}=\A$ and
\begin{align*}
   \Omegag{k} = \{a_0da_1da_2\cdots da_k: a_i\in\A\text{ for }i=0,\ldots,k\}
\end{align*}
for $k\geq 1$. It is straightforward to check that $\Omegag{}=\bigoplus_{k\geq 0}\Omegag{k}$ is a differential graded subalgebra of $\Omegabg{}$ as well as a sub $\ast$-bimodule of $\Omegabg{}$. The differential graded algebra $\Omegag{}$ more closely resembles the algebra of differential forms on a manifold, where the differential forms are generated by the algebra of smooth functions.

\section{Derivation based calculi}\label{sec:derivations.based.calculi}

\subsection{Connections and hermitian forms}
In this section, we will introduce well-known notions of connections and hermitian forms in a notation which is adapted to the context of derivation based differential calculi. As we will be using homomorphims defined on the Cartesian products of left and right modules, we start by introducing the following notation.

\begin{definition}\label{def:HomAB}
    Let $\A,\B$ be rings and let $M$ be a left $\A$-module, let $N$ a right $\B$-module and let $S$ be a $(\A,\B)$-bimodule. Define
    \begin{align*}
        \Hom_{\A,\B}(M\times N,S)
    \end{align*}
    as the set of biadditive maps $f:M\times N\to S$ such that $f(am,nb) = af(m,n)b$
    for $m\in M$, $n\in N$, $a\in\A$ and $b\in\B$. Moreover, if $\A=\B$ then we write 
    \begin{align*}
        \Hom_{\A}(M\times N,S)\equiv\Hom_{\A,\A}(M\times N,S).
    \end{align*}
\end{definition}

\noindent
In the above notation, we shall freely consider left/right modules over a
commutative ring $\A$ as $\A$-bimodules.
Now, let $\A$ be a $\ast$-algebra and let $M$ be a left $\A$-module. One can define a right module structure on $M$ by setting
\begin{align*}
    m\cdot a = a^\ast m,
\end{align*}
for $a\in\A$ and $m\in M$, and the corresponding right module is denoted by $\Mh$ and called the conjugate module (note that, in general, $M$ is \emph{not} a bimodule with respect to the right module structure introduced above.). Similarly, for a right $\A$-module $M$
the conjugate module $\Mh$ is a left $\A$-module with $a\cdot m=ma^\ast$.

Let us now recall the concept of a connection on a left $\A$-module.
\begin{definition}\label{def_connection}
    Let $\A$ be a $\ast$-algebra and let $\g\subseteq\Der(\A)$ be a Lie algebra. A left $\g$-connection on a left $\A$-module $M$ is a map $\nabla:M\times\g\to M$ such that
    \begin{enumerate}
        \item
        $\nabla_\d(m+m')=\nabla_\d m+\nabla_\d m'$,
        
        \item
        $\nabla_{\d+\d'}m=\nabla_\d m+\nabla_{\d'}m$,
        
        \item
        $\nabla_\d(am)=a(\nabla_\d m)+(\d a)m$,
    \end{enumerate}
    for $m,m'\in M$, $\d,\d'\in\mathfrak{g}$, $a\in\mathcal{A}$, and
    \begin{enumerate}
        \setcounter{enumi}{3}
        \item $\nabla_{z\d}m=z\nabla_\d m$
    \end{enumerate}
    for $m\in M$, $\d\in\g$ and $z\in Z(\A)$ such that $z\d\in\g$.
\end{definition}

\noindent
When there is no risk of confusion, we shall often leave out the explicit reference to the Lie algebra $\g$ and simply say that $\nabla$ is a left connection on $M$. 

\begin{remark}
    The conditions in Definition~\ref{def_connection} implies that  $\g$-connections are elements of $\Hom_{\complex,\ZhA}(M\times\g,M)$.
\end{remark}

\noindent
As introduced above, given $\g\subseteq\Der(\A)$, it might happen that there exist no $\g$-connections on the module $M$. These modules are clearly not interesting from the point of view of finding Levi-Civita connections, and therefore we make the following definition.

\begin{definition}
    Let $\A$ be a $\ast$-algebra and let $\g\subseteq\Der(\A)$ be a Lie algebra. A left $\A$-module $M$ is called a left $\g$-connection module if there exists a left $\g$-connection $\nabla:M\times\g\to M$.
\end{definition}

\noindent
We denote the set of all $\g$-connections on $M$ by $\Cg(M)$.
Given two connections $\nabla,\nabla'\in\Cg(M)$ one notes that their difference 
$\alpha = \nabla-\nabla'$ satisfies
\begin{align*}
    &\alpha(m,\d+\d') = \alpha(m,\d) + \alpha(m,\d')\qquad
    \alpha(m+m',\d)=\alpha(m,\d)+\alpha(m'\d)\\
    &\alpha(am,\d)= a\alpha(m,\d)\qquad
    \alpha(m,z\d)=z\alpha(m,\d)
\end{align*}
for $a\in\A$, $m,m'\in M$, $\d,\d'\in\g$ and $z\in Z(\A)$ such that $z\cdot \d\in\g$. In other words, 
$\alpha\in\Hom_{\A,\ZhA}(M\times \g,M)$.
Conversely, given a connection $\nabla\in\Cg(M)$ and $\alpha\in\Hom_{\A,\ZhA}(M\times \g,M)$ it is easy to check that $\nabla+\alpha$ is a connection on $M$.
That is, the set $\Hom_{\A,\ZhA}(M\times \g,M)$ parametrizes all $\g$-connections on the $\g$-connection module $M$. More precisely, if $M$ is a $\g$-connection module and $\nabla^0\in\Cg(M)$ then $\phi_{\nabla^0}:\Hom_{\A,\ZhA}(M\times\g,M)\to\Cg(M)$ defined by
\begin{align*}
    \phi_{\nabla^0}(\alpha) = \nabla^0+\alpha
\end{align*}
is a bijection.

In noncommutative geometry, one considers (finitely generated projective) modules as analogues of vector bundles in differential geometry. Vector bundles can be equipped with hermitian metrics, and a natural noncommuatative generalization is that of hermitian forms. 

\begin{definition}\label{def:hermitian.form}
    Let $\A$ be a $\ast$-algebra and let $M$ be a left $\mathcal{A}$-module. A left hermitian form on $M$ is a map $h:M\times M\to\mathcal{A}$ such that
    \begin{enumerate}
        \item\label{eq:hform.additive}
        $h(m_1+m_2,m_3)=h(m_1,m_3)+h(m_2,m_3)$,

        \item\label{eq:hform.left.linear}
        $h(am_1,m_2,)=ah(m_1,m_2)$,

        \item\label{eq:hform.symmetry}
        $h(m_1,m_2)^\ast=h(m_2,m_1)$,
    \end{enumerate}
    for $m_1,m_2,m_3\in M$ and $a\in\A$. Furthermore, $h$ is called invertible if $\hh:M\to M^\ast$, defined by $\hh(m_1)(m_2) = h(m_2,m_1)$ for $m_1,m_2\in M$, is a bijection. An invertible left hermitian form on $M$ is called a left hermitian metric on $M$. 
\end{definition}

\begin{remark}
    Similarly, a right hermitian form $h$ on a right $\A$-module $M$ satisfies \eqref{eq:hform.additive} and \eqref{eq:hform.symmetry} in Definition~\ref{def:hermitian.form} together with $h(m_1,m_2a) = h(m_1,m_2)a$.
\end{remark}

\begin{remark}
    In the notation of Definition~\ref{def:HomAB}, left hermitian forms on a left $\A$-module $M$ are elements of $\Hom_{\A}(M\times\Mh,\A)$ satisfying $h(m_1,m_2)^\ast=h(m_2,m_1)$.    
\end{remark}

\noindent
Next, let us note a few properties of invertible hermitian forms.

\begin{proposition}\label{prop:h.properties}
    Let $h$ be a left hermitian form on a left $\A$-module $M$. Then
    \begin{align*}
        \hh(am)=\hh(m)a^\ast,
    \end{align*}
    for $m\in M$ and $a\in\A$. Moreover, if $h$ is invertible then
    \begin{align*}
        \hhi(\phi a)=a^\ast\hhi(\phi)
    \end{align*}
    for $\phi\in M^\ast$ and $a\in\A$.
\end{proposition}

\begin{proof}
    Let $m\in M$ and $a\in\A$. Then for any $n\in M$
    \[
    \hh(am)(n)=h(n,am)=h(n,m)a^\ast=\hh(m)(n)a^\ast.
    \]
    Hence $\hh(am)=\hh(m)a^\ast$. Also, since for any $\phi\in M^\ast$, $\hhi(\phi)\in M$, it follows that
    \[
    \phi a=\hh(\hhi(\phi))(a^\ast)^\ast=\hh(a^\ast\hhi(\phi))\equivalent \hhi(\phi a)=a^\ast\hhi(\phi).\qedhere
    \]
\end{proof}

\begin{remark}
    Note that Proposition~\ref{prop:h.properties} implies that $\hh$ is a left module homomorphism from $M$ to $\Mh^\ast$ (the conjugate of the dual module).
\end{remark}

\noindent
In case $h$ is invertible, one may define a hermitian form on the dual module as follows.

\begin{proposition}
    Let $h$ be an invertible left hermitian form on the left $\A$-module $M$. Define $h^{-1}:M^\ast\times M^\ast\to\A$ as
    \begin{align}\label{eq:def.h.inverse}
        h^{-1}(\phi_1,\phi_2) = h\paraa{\hhi(\phi_1),\hhi(\phi_2)}
    \end{align}
    for $\phi_1,\phi_2\in M^\ast$. Then $h^{-1}$ is a right hermitian form on the right $\A$-module $M^\ast$.
\end{proposition}

\begin{proof}
    Since $\hhi$ is additive and $h$ is biadditive, it is clear that $\hi$ is biadditive. Since $h(m_1,m_2)^\ast=h(m_2,m_1)$ it immediately follows that $\hi(\phi_1,\phi_2)^\ast=\hi(\phi_2,\phi_1)$. Moreover,
    \begin{align*}
        \hi(\phi_1,\phi_2a) = h\paraa{\hhi(\phi_1),\hhi(\phi_2a)}
        =h\paraa{\hhi(\phi_1),a^\ast\hhi(\phi_2)} 
        = \hi(\phi_1,\phi_2)a
    \end{align*}
    showing that $\hi$ is a hermitian form on $M^\ast$.
\end{proof}

\noindent 
For future use, let us also note that 
\begin{align}
    \hi(\phi_1,\phi_2) = h\paraa{\hhi(\phi_1),\hhi(\phi_2)}
    =\hh\paraa{\hhi(\phi_2)}\paraa{\hhi(\phi_1)} 
    = \phi_2\paraa{\hhi(\phi_1)}. 
\end{align}
In differential geometry, a connection can preserve the metric on a vector bundle, and in the current setup one formulates this property as compatibility with a hermitian form. 

\begin{definition}
    Let $\A$ be a $\ast$-algebra and let $\g\subseteq\Der(\A)$ be a Lie algebra. Moreover, let $M$ be a left $\A$-module with a left hermitian form $h$. A left $\g$-connection $\nabla$ on $M$ is compatible with $h$ if
    \begin{align}
        \d h(m_1,m_2)=h(\nabla_\d m_1,m_2)+h(m_1,\nabla_{\d^\ast}m_2) \label{eq:metric.condition}   
    \end{align}
    for $m_1,m_2\in M$ and $\d\in\mathfrak{g}$. The set of $\g$-connections on $M$ compatible with $h$ will be denoted by $\Cg(M,h)$.
\end{definition}

\subsection{Derivation based calculi}

In classical geometry, a differentiable structure on a manifold is defined by a choice of an open covering together with local trivialization satisfying ceratin properties. In noncommutative geometry, where the starting point is the algebra of functions, rather than the underlying set, the analogue of a differentiable structure can be thought of in different ways. A common way of providing such a structure is by choosing a differential graded algebra with the algebra of functions in degree zero; this provides an analogue of differentiable forms. 

In a derivation based calculus, as we have seen in the previous sections, one defines a differential graded algebra by choosing a Lie algebra of derivations. Moreover, for convenience, as our main interest lies in connections, we shall also assume that there exist connections on $\Omegaoneg$. Hence, we make the following definition.

\begin{definition}
    A (left) derivation based calculus is a pair $(\A,\g)$ where $\A$ is a unital $\ast$-algebra over $\complex$ and $\g\subseteq\Der(\A)$ is a $\ast$-closed Lie algebra such that $\Omegaoneg$ is a (left) $\g$-connection module. 
\end{definition}

\begin{definition}
    A (left) derivation based calculus $(\A,\g)$ is called finitely generated if $\Omegaoneg$ is a finitely generated module, and projective if $\Omegaoneg$ is a projective (left) module.
\end{definition}

\noindent
With a view towards torsion free connections compatible with a hermitian form, we introduce hermitian calculi as derivation based calculi equipped with an invertible hermitian form.

\begin{definition}
    A (left) hermitian calculus is a triple $(\A,\g,h)$ such that $(\A,\g)$ is a (left) derivation based calculus and $h$ is an invertible (left) hermitian form on $\Omegaoneg$. 
\end{definition}

\subsection{Generalized metric symmetry}\label{sec:gen.metric.symmetry}

\noindent
Let us now compare our current setup with classical Riemannian geometry, and introduce a generalized symmetry condition for the hermitian form, which will later turn out to be necessary for finding Levi-Civita connections on hermitian calucli.

To this end, let $(M,g)$ be a Riemannian manifold. In this case, $\A=C^\infty(M,\complex)$ (with $\ast$-algebra structure given by pointwise complex conjugation) and
$\g=\Der(\A)$ is isomorphic to the set of complexified smooth vector fields $\X_{\complex}(M)$. Moreover, it follows that $\Omegaoneg=\Omegaonebg=\Omega^1_{\complex}(M)$, i.e. the complexification of the set of differential forms on $M$. It is clear that 
\begin{align*}
 \paraa{C^\infty(M,\complex),\Der(C^{\infty}(M,\complex))}
\end{align*}
is a derivation based calculus since there exist connections on the cotangent bundle of $M$.
The Riemannian metric $g$ is extended linearly to the complexification $\X_{\complex}(M)$ as
\begin{align*}
    g_\complex(X_1+iX_2,Y_1+iY_2) = g(X_1,Y_1)-g(X_2,Y_2)+i\paraa{g(X_1,Y_2)+g(X_2,Y_1)}
\end{align*}
for $X_1,X_2,Y_1,Y_2\in\X(M)$, satisfying $g_\complex(Z,W)^\ast=g_\complex(Z^\ast,W^\ast)$ for $Z,W\in\X_\complex(M)$.

Let $\nabla$ be the Levi-Civita connection on $(M,g)$. It can be extended to $\X_\complex(M)$ as
\begin{align*}
    \nabla_{X_1+iX_2}(Y_1+iY_2) = \nabla_{X_1}Y_1-\nabla_{X_2}Y_2
    +i\paraa{\nabla_{X_1}Y_2+\nabla_{X_2}Y_1}
\end{align*}
for $X_1,X_2,Y_1,Y_2\in\X(M)$, satisfying $(\nabla_Z W)^\ast=\nabla_{Z^\ast}W^\ast$ and
\begin{align*}
    &\nabla_Z W-\nabla_W Z-[Z,W]=0\\
    &Z\paraa{g_\complex(W_1,W_2)} = g_\complex\paraa{\nabla_{Z}W_1,W_2}+g_\complex\paraa{W_1,\nabla_Z W_2}
\end{align*}
for $Z,W,W_1,W_2\in\X_\complex(M)$. Next, one defines a connection on $\Omega^1_\complex(M)$ as
\begin{align*}
    (\nabla_{Z}\omega)(W) =Z\paraa{\omega(W)}-\omega\paraa{\nabla_Z W} 
\end{align*}
for $Z,W\in\X_\complex(M)$ and $\omega\in\Omega^1_\complex(M)$, satisfying 
$(\nabla_{Z}\omega)^\ast=\nabla_{Z^\ast}\omega^\ast$ and
\begin{align*}
    (\nabla_{Z_1}\omega)(Z_2)-(\nabla_{Z_2}\omega)(Z_1)-d\omega(Z_1,Z_2)=0
\end{align*}
for $Z_1,Z_2\in\X_\complex(M)$ and $\omega\in\Omega^1_\complex(M)$. Thus, $\nabla$ is a torsion free connection on $\Omega^1_\complex(M)$.

Since $g_\complex$ is nondegenerate it defines an invertible hermitian form $h$ on $\Omega^1_\complex(M)$ as
\begin{align}\label{eq:h.g.inv.def}
    h(\omega,\eta) = g_{\complex}^{-1}(\omega,\eta^\ast),
\end{align}
where $g_{\complex}^{-1}$ denotes the inverse (complexified) metric on $\Omega^1_{\complex}(M)$, and it is straightforward to show that
\begin{align*}
    Z\paraa{h(\omega,\eta)} = h\paraa{\nabla_Z\omega,\eta}+h\paraa{\omega,\nabla_{Z^\ast}\eta}.
\end{align*}
Hence, the Levi-Civita connection on $(M,g)$ induces a Levi-Civita connection on the hermitian calculus
\begin{align*}
    \paraa{C^\infty(M,\complex),\Der(C^{\infty}(M,\complex)),h}.
\end{align*}
Now, if $h$ is a hermitian form on $\Omega^1_{\complex}(M)$ induced from a metric as in \eqref{eq:h.g.inv.def} then $h^{-1}$ is a hermitian form on $\X_\complex(M)$ satisfying 
\begin{align*}
    h^{-1}(Z,W) = g_{\complex}(Z^\ast,W).
\end{align*}
Since $g_\complex$ is symmetric, it follows that $h^{-1}$ has an additional symmetry, namely
\begin{align*}
    0 = g_\complex(Z,W) - g_{\complex}(W,Z) = h^{-1}(Z^\ast,W)-h^{-1}(W^\ast,Z),
\end{align*}
which is not in general satisfied by a hermitian form and we introduce $\rho\in\Omega^2_{\complex}(M)$ as
\begin{align}\label{eq:rho.def.geom}
    \rho(Z,W) = h^{-1}(Z^\ast,W)-h^{-1}(W^\ast,Z)
\end{align}
measuring the failure of the symmetry property. If $h^{-1}$ is induced by a Riemannian metric as above, then $\rho=0$. In a noncommutative setting, demanding that $\rho=0$ is in general quite restrictive. However, it turns out that $d\rho=0$ is a suitable generalization of the symmetry condition in the sense that it will be a necessary condition for the existence of Levi-Civita connections (cf. Section~\ref{sec:lc.connections} and Proposition~\ref{prop:LC.necessary.cond}). 

\begin{definition}\label{def:weakly.symmetric}
    Let $(\A,\g,h)$ be a left hermitian calculus. The symmetry form $\rho\in\Omegabg{2}$ is defined as
    \begin{align}\label{eq:def.symmetric.form}
        \rho(\d_1,\d_2) = h^{-1}\paraa{\varphi(\d_1^\ast),\varphi(\d_2)}
    -h^{-1}\paraa{\varphi(\d_2^\ast),\varphi(\d_1)}
    \end{align}
    for $\d_1,\d_2\in\g$. If $\rho=0$ then $(\A,\g,h)$ is called symmetric and if $d\rho=0$ then $(\A,\g,h)$ is called weakly symmetric.
\end{definition}

\begin{remark}
    It is easy to check that $\rho$ is indeed an element of $\Omegabg{2}$:
    the antisymmetry is immediate and the partial $\ZA$-linearity follows from the (skew)linearity of $\varphi$ and $h^{-1}$. We also note that $\rho^\ast=\rho$.
\end{remark}

\begin{remark}
    Note that there is a close similarity between the concept of a symmetric hermitian calculus and the concept
    of quantum symmetry in \cite{bm:Quantum.Riemannian.geometry}. In that context, a metric is an (invertible) element $g\in\Omegaoneg\otimes_{\A}\Omegaoneg$ and corresponds to $h^{-1}$ (when taking the conjugation of one of the arguments into account). Then the quantum symmetry condition $\wedge g=0$ is analogous to \eqref{eq:def.symmetric.form}.
\end{remark}

\noindent
As in differential geometry, one notes that if $\dim(\g)\leq 2$ then $\Omegabg{3}=0$, implying that every symmetry form is closed and, hence, that $(\A,\g,h)$ is a weakly symmetric hermitian calculus. 

\section{Index calculus on projective modules}\label{sec:ind.cal.proj.mods}

\noindent 
Finitely generated projective modules are of particular interest in noncommutative geometry. Apart from having general algebraic features that make them tractable, they correspond to vector bundles in the classical geometric setting. One of the key features of finitely generated projective modules is the fact that there exists a dual basis, which allows one to do many computations with respect to a set of generators. In this section we collect a few computational formulas for hermitian forms on projective modules with respect to an arbitrary choice of generators, which will be of later use in the paper. (Note that similar formulas appear in the litterature in various formulations, see e.g. \cite{bm:starCompatibleConnections}.) 

To this end, assume that $M$ is a finitely generated projective left $\A$-module and let 
$\{\theta^i\}_{i=1}^N$ be a set of generators of $M$. Recall that a dual basis $\{\phi_i\}_{i=1}^N$ is a set of elements in $M^\ast$ such that
\begin{align*}
    m =\sum_{i=1}^N \phi_i(m)\theta^i\equiv \phi_i(m)\theta^i
    \qqand
    f = \sum_{i=1}^N\phi_if(\theta^i)\equiv\phi_i\psi(\theta^i)
\end{align*}
for all $m\in M$ and $f\in M^\ast$. As indicated above, we will in the following make use of a summation convention where repeated indices are summed over the appropriate range.

Let $\A^N$ denote a free module of rank $N$ with basis $\{e^i\}_{i=1}^N$; moreover, since the module is free there exist $\{e_i\}_{i=1}^N\subseteq (\A^N)^\ast$ such that $e_i(e^k)=\delta_i^k\mid$. Define $\pi:\A^N\to M$ by $\pi(e^i)=\theta^i$. Choosing a dual basis in $M^\ast$ is equivalent to choosing a section $s:M\to\A^N$ 
\begin{equation*}
    \begin{tikzcd}[every arrow/.append style={shift left}]
        \A^N \arrow{r}{\pi} & M \arrow{l}{{s\vphantom{1}}} 
    \end{tikzcd}
\end{equation*}
satisfying $\pi\circ s=\id_M$. Namely, given a section $s:M\to\A^N$ one considers
\begin{align*}
    s^\ast:(\A^N)^\ast\to M^\ast\qquad
    s^\ast(f)(m) = f\paraa{s(m)}
\end{align*}
for $f\in(\A^N)^\ast$ and $m\in M$, and defines $\phi_i = s^\ast(e_i)$. It follows that
\begin{align*}
    \phi_i(m)\theta^i = s^\ast(e_i)(m)\theta^i = e_i(s(m))\pi(e^i)
    =\pi\paraa{e_i(s(m))e^i} = \pi(s(m)) = m
\end{align*}
for $m\in M$, since $e_i(U)e^i=U$ for all $U\in\A^N$. Hence, $\{\phi_i\}_{i=1}^N$ is a dual basis. Conversely, given a dual basis $\{\phi_i\}_{i=1}^N$ one sets $s(m) = \phi_i(m)e^i$ and checks that
\begin{align*}
    (\pi\circ s)(m) = \pi\paraa{\phi_i(m)e^i} = \phi_i(m)\pi(e^i) = \phi_i(m)\theta^i = m.
\end{align*}
for all $m\in M$.

Furthermore, one defines $p=s\circ\pi:\A^N\to\A^N$, giving 
\begin{align*}
    p(e^i) = s\paraa{\pi(e^i)} = s(\theta^i) = \phi_k(\theta^i)e^k,
\end{align*}
and it follows that $p^2=p$ and $M\simeq p(\A^N)$.

Next, let us assume that $h$ is an invertible hermitian form on $M$ and set
\begin{align*}
    h^{ij} = h(\theta^i,\theta^j) \qqand
    h_{ij} = h^{-1}(\phi_i,\phi_j).
\end{align*}
Noting that 
\begin{align*}
    \hh(m)(m') = \hh\paraa{\phi_i(m')\theta^i,\phi_k(m)\theta^k}
    =\phi_i(m')h^{ik}\phi_k(m)^\ast
\end{align*}
one concludes that $\hh(m)=\phi_i h^{ik}\phi_k(m)^\ast$ for $m\in M$. Similarly, one finds that
\begin{align*}
    \hhi(f) = \phi_i\paraa{\hhi(f)}\theta^i=h^{-1}(f,\phi_i)\theta^i
    =h^{-1}\paraa{\phi_kf(\theta^k),\phi_i}\theta^i = f(\theta^k)^\ast h_{ki}\theta^i
\end{align*}
for $f\in M^\ast$.
The invertibility of $\hh$ may be expressed as
\begin{align*}
    \theta^i &= \hhi(\hh(\theta^i)) = \hhi\paraa{\phi_j h^{jk}\phi_k(\theta^i)^\ast}
    =\hhi\paraa{\phi_j h(\theta^j,\phi_k(\theta^i)\theta^k}
    =\hhi(\phi_jh^{ji})\\
    &=h^{ij}\hhi(\phi_j)
    =h^{ij}\phi_j(\theta^k)^\ast h_{kl}\theta^l
    =h^{ij}\hi\paraa{\phi_k\phi_j(\theta^k),\phi_l}\theta^l=h^{ij}h_{jl}\theta^l
\end{align*}
and, similarly, one finds that $\phi_i=\phi_jh^{jk}h_{ki}$ as a consequence of $\phi_i=\hh(\hhi(\phi_i))$. For convenience, we introduce 
\begin{align*}
    \theta_i = \hhi(\phi_i) = h_{ij}\theta^j\qand
    \phi^i = \hh(\theta^i) = \phi_jh^{ji},
\end{align*}
implying that
\begin{align*}
    \hh(m) = \phi^i\phi_i(m)^\ast\qand
    \hhi(f) = f(\theta^i)^\ast\theta_i
\end{align*}
for $m\in M$ and $f\in M^\ast$. 
As stated previously, $p=s\circ\pi:\A^N\to\A^N$ is a projection such that $p(\A^n)\simeq M$. Let us now show that this projection can also be expressed in terms of the hermitian form $h$. Namely, define $p_h:\A^N\to\A^N$ by
\begin{align*}
    p_h(U_ie^i) = U_ih^{ij}h_{jk}e^k.
\end{align*}
It follows that 
\begin{align*}
    p_h(U_ie^i) = U_ih^{ij}h_{jk}\theta^k
    &=U_ih\paraa{\phi_l(\theta^i)\theta^l,\theta^j}h_{jk}\theta^k
    =U_i\phi_l(\theta^i)h^{lj}h_{jk}\theta^k \\
    &=U_i\paraa{\phi_lh^{lk}h_{jk}}(\theta^i)\theta^k
    =U_i\phi_k(\theta^i)\theta^k = p(U_ie^i),
\end{align*}
showing that $p_h(U)=p(U)$.

Now, let us turn to the case when $M=\Omegaoneg$. One writes
\begin{align*}
    \varphi(\d) = \phi_i\varphi(\d)(\theta^i) = \phi_i\theta^i(\d)
\end{align*}
giving 
\begin{align*}
    \hi\paraa{\varphi(\d_1^\ast),\varphi(\d_2)}
    &=\hi\paraa{\phi_i\theta^i(\d_1^\ast),\phi_j\theta^j(\d_2)}
    =(\theta^i)^\ast(\d_1)h_{ij}\theta^j(\d_2)\\
    &=(h_{ji}\theta^i)^\ast(\d_1)\theta^j(\d_2)
    =\theta_i^\ast(\d_1)\theta^i(\d_2).
\end{align*}
The above computation gives the following expression for the symmetry form in terms of the generators:
\begin{align*}
    \rho(\d_1,\d_2) 
    &= h^{-1}\paraa{\varphi(\d_1^\ast),\varphi(\d_2)}
    -h^{-1}\paraa{\varphi(\d_2^\ast),\varphi(\d_1)}\\
    &= \theta_i^\ast(\d_1) \theta^i(\d_2)-\theta_i^\ast(\d_2) \theta^i(\d_1)
    =(\theta_i^\ast\theta^i)(\d_1,\d_2)
\end{align*}
that is, $\rho=\theta_i^\ast \theta^i$.

\begin{lemma}\label{lemma:rho.in.generators}
    Let $(\A,\g,h)$ be a finitely generated projective hermitian calculus and let $\{\theta^i\}_{i=1}^N$ be a set of generators for $\Omegaoneg$. Then
    \begin{align}\label{eq:drho.in.generators}
        &d\rho = \theta_i^\ast (dh^{ij})\theta_j+(d\theta^i)^\ast\theta_i-\theta_i^\ast d\theta^i.
    \end{align}
\end{lemma}

\begin{proof}
    Since $\rho=\theta_i^\ast\theta^i$ one finds that
    \begin{align*}
        d\rho 
        &= d\paraa{\theta_i^\ast\theta^i}
        = d\paraa{(\theta^j)^\ast h_{ji}\theta^i}
        =(d\theta^j)^\ast \theta_j-(\theta^j)^\ast d(h_{ji}\theta^i)\\
        &=(d\theta^j)^\ast \theta_j-(\theta^j)^\ast\paraa{(dh_{ji})\theta^i+h_{ji}d\theta^i}\\
        &=-(\theta^j)^\ast(dh_{ji})\theta^i+(d\theta^j)^\ast \theta_j-\theta_i^\ast d\theta^i.
    \end{align*}
    Rewriting the first term gives
    \begin{align*}
        (\theta^j)^\ast(dh_{ji})\theta^i
        &=\theta_k^\ast h^{kj}(dh_{ji})h^{il}\theta_l\\
        &=\theta_k^\ast d(h^{kj}h_{ji}h^{il})\theta_l
        -\theta_k^\ast(dh^{kj})h_{ji}h^{il}\theta_l
        -\theta_k^\ast h^{kj}h_{ji}(dh^{il})\theta_l\\
        &=\theta_k^\ast(dh^{kl})\theta_l - \theta_k^\ast(dh^{kj})\theta_j - \theta_i^\ast(dh^{il})\theta_l
        =-\theta_i^\ast(dh^{ij})\theta_j
    \end{align*}
    implying that
    \begin{align*}
        d\rho
        = \theta_i^\ast(dh^{ij})\theta_j+(d\theta^j)^\ast \theta_j-\theta_i^\ast d\theta^i,
    \end{align*}
    proving \eqref{eq:drho.in.generators}.
\end{proof}

\noindent
For easy reference, let us compile a list of the most relevant formulas: 
\begin{alignat}{2}
    &h^{ij}h_{jk}\theta^k = \theta^i
    &&\phi_i h^{ij}h_{jk} = \phi_k\\
    &\hh(m) = \phi_i\phi^i(m)^\ast
    &\qquad&\hhi(f)=f(\theta^i)^\ast \theta_i\\
    &\rho = \theta_i^\ast\theta^i
    &&d\rho = \theta_i^\ast (dh^{ij})\theta_j+(d\theta^i)^\ast\theta_i-\theta_i^\ast d\theta^i,
\end{alignat}
for $m\in M$ and $f\in M^\ast$.

\section{Connections compatible with hermitian forms}\label{sec:conn.comp.herm}

\noindent
Having all the basic definitions in place, we proceed to study the existence of connections compatible with a hermitian form. In the context of derivation based calculi, the existence of compatible connections on projective modules is well-known (see e.g. \cite{a:levi-civita.class.nms} for a result closely related to the current situation). However, for our purposes, we would like to present a formulation which is not dependent on the projectivity of the module and which is more adapted to our context. We start by introducing the necessary notation.

Let $M$ be a left $\A$-module. Given a left hermitian form $h$ and 
\begin{align*}
    \alpha\in\Hom_{\complex,\ZhA}(M\times\g,M)    
\end{align*}
we define $h_{\alpha}(m_1,m_2)\in\Omegabg{1}$ as
\begin{align*}
    &h_{\alpha}(m_1,m_2)(\d) = h\paraa{\alpha(m_1,\d),m_2} 
\end{align*}
for $m_1,m_2\in M$ and $\d\in\g$, implying that $h_\alpha\in\Hom_{\complex,\A}(M\times \Mh,\Omegabg{1})$. Moreover, if $\alpha\in\Hom_{\A,\ZhA}(M\times\g,M)$ then $h_{\alpha}\in\Hom_{\A}(M\times\Mh,\Omegabg{1})$.
Furthermore, define $dh=d\circ h\in\Hom_{\complex}(M\times\Mh,\Omegabg{1})$, giving 
\begin{align*}
    (dh)(m_1,m_2)(\d) = d(h(m_1,m_2))(\d)=\d h(m_1,m_2),
\end{align*}
and for $L\in\Hom_{\complex}(M\times\Mh,\Omegabg{1})$, define $L^\ast\in\Hom_{\complex}(M\times\Mh,\Omegabg{1})$ as
\begin{align*}
    L^\ast(m_1,m_2) = L(m_2,m_1)^\ast,
\end{align*}
from which it immediately follows that $(L^\ast)^\ast=L$. One checks that
\begin{align*}
    (dh)^\ast(m_1,m_2)(\d) &= \paraa{dh(m_2,m_1)}^\ast(\d)
    =\paraa{dh(m_2,m_1)(\d^\ast)}^\ast\\
    &= \paraa{\d^\ast h(m_2,m_1)}^\ast
    =\d h(m_1,m_2) = (dh)(m_1,m_2)(\d),
\end{align*}
that is, $(dh)^\ast=dh$. 

Recall that a (left) connection $\nabla$ on $M$ is compatible with $h$ if 
\begin{align}\label{eq:h.comp.again}
        \d h(m_1,m_2)=h(\nabla_\d m_1,m_2)+h(m_1,\nabla_{\d^\ast}m_2)
\end{align}
for $\d\in\g$ and $m_1,m_2\in M$. In the notation introduced above one can rewrite \eqref{eq:h.comp.again} as
\begin{align}
    dh = h_{\nabla} + h_{\nabla}^\ast.
\end{align}
In the following, we will prove that given an arbitrary connection on $M$ and an invertible hermitian form, one can construct a connection that is compatible with $h$. Let us start with the following lemma.

\begin{lemma}\label{lemma_phi.alpha.L}
    Let $(\A,\g)$ be a derivation based calculus and let $h$ be an invertible left hermitian form on the left $\A$-module $M$.
    Moreover, for $L\in\Hom_{\complex,\A}(M\times\Mh,\Omegabg{1})$, $m\in M$ and $\d\in\g$, define
    \begin{align*}
        &\phi^L_{m,\d}(m') = \paraa{L(m,m')(\d)}^\ast
    \end{align*}
    for $m'\in M$. Then $\phi^L_{m,\d}\in M^\ast$ and $L=h_{\alpha_L}$ with $\alpha_L(m,\d) = \hhi(\phi^L_{m,\d})$.
\end{lemma}

\begin{proof}
    Let us first show that $\phi^L_{m,\d}\in M^\ast$. It is clear that $\phi^L_{m,\d}$ is additive since $L$ is additive in both arguments. Furthermore, for $a\in\A$ and $m'\in M$ one computes
    \begin{align*}
        \phi^L_{m,\d}(am') &= L(m,am')(\d)^\ast = L(m,m'\cdot a^\ast)(\d)^\ast
        =\paraa{L(m,m')a^\ast}(\d)^\ast\\
        &= \paraa{L(m,m')(\d)a^\ast}^\ast
        =aL(m,m')(\d)^\ast = a\phi^L_{m,\d}(m')
    \end{align*}
    showing that $\phi^L_{m,\d}$ is a left module homomorphism. Hence, $\phi^L_{m,\d}\in M^\ast$. Moreover, we note that for $\lambda\in\complex$, $\d\in\g$ and $z\in\ZA$ such that $z\d\in\g$
    \begin{align*}
        \phi^L_{\lambda m,z\d}(m') = L(\lambda m,m')(z\d)^\ast = \bar{\lambda}\phi^L_{m,\d}(m')z^\ast
        =(\bar{\lambda}\phi^L_{m,\d}z^\ast)(m')
    \end{align*}
    giving 
    \begin{align*}
        \alpha_L(\lambda m,z\d) = \hhi(\phi^L_{\lambda m,z \d})
        =\hhi(\bar{\lambda}\phi^L_{m,\d}z^\ast) = \lambda z\alpha_L(m,\d)
    \end{align*}
    showing that $\alpha_L\in\Hom_{\complex,\ZhA}(M\times\g,M)$. Finally, one computes
    \begin{align*}
        h_{\alpha_L}(m,m')(\d)
        &= h\paraa{\alpha_L(m,\d),m'}
        =h\paraa{m',\alpha_L(m,\d)}^\ast
        =\hh\paraa{\alpha_L(m,\d)}(m')^\ast\\
        &= \phi^L_{m,\d}(m')^\ast = L(m,m')(\d)
    \end{align*}
    showing that $h_{\alpha_L}=L$.
\end{proof}

\noindent
Using Lemma~\ref{lemma_phi.alpha.L} one proves next that if $L$ fulfills a certain derivation property (cf. \eqref{eq:L.left.arg.deriv} below) and $L+L^\ast=dh$ then $\alpha_L$ defines a connection compatible with $h$.

\begin{proposition}\label{prop:L.hhi.connection}
    Let $(\A,\g)$ be a derivation based calculus and let $h$ be an invertible hermitian form on a left $\A$-module $M$. If $L\in\Hom_{\complex,\A}(M\times\Mh,\Omegabg{1})$ such that
    \begin{align}\label{eq:L.left.arg.deriv}
        L(am_1,m_2) = aL(m_1,m_2) + (da)h(m_1,m_2)
    \end{align}
    for $a\in\A$ and $m_1,m_2\in M$ then
    \begin{align*}
        \nabla_\d m = \hhi(\phi^L_{m,\d})
    \end{align*}
    is a connection on $M$. Furthermore, if $L+L^\ast=dh$ then $\nabla$ is compatible with $h$.
\end{proposition}

\begin{proof}
    It is clear that $\nabla$ is a additive in both arguments. First, we note that if $L$ satisfies \eqref{eq:L.left.arg.deriv} then 
    \begin{align*}
        \phi^L_{am,\d}(m') 
        &= L(am,m')(\d)^\ast 
        = L(m,m')(\d)^\ast a^\ast + h(m',m)(\d a)^\ast\\
        &=(\phi^L_{am,\d}a^\ast)(m')+\paraa{\hh(m)(\d a)^\ast}(m')
    \end{align*}
    giving
    \begin{align*}
        \nabla_{\d}(am) 
        &= \hhi\paraa{\phi^L_{am,\d}}
        =\hhi\paraa{\phi^L_{m,\d}a^\ast + \hh(m)(\d a)^\ast}\\
        &= a\hhi\paraa{\phi^L_{m,\d}}+(\d a)m
        =a\nabla_{\d}m + (\d a)m.
    \end{align*}
    For $z\in\ZA$ and $\d\in\g$ such that $z\d\in\g$ one checks that
    \begin{align*}
        \phi^L_{m,z\d}(m') = L(m,m')(z\d)^\ast = z^\ast L(m,m')(\d)^\ast
        =L(m,m')(\d)^\ast z^\ast = \paraa{\phi^L_{m,\d}z^\ast}(m')
    \end{align*}
    implying that
    \begin{align*}
        \nabla_{z\d}m = \hhi\paraa{\phi^L_{m,z\d}}
        =\hhi\paraa{\phi^L_{m,\d}z^\ast} = z\hhi\paraa{\phi^L_{m,\d}} = z\nabla_{\d}m
    \end{align*}
    showing that $\nabla$ is indeed a connection on $M$.

    Next, assume that $L+L^\ast=dh$. It follows from Lemma~\ref{lemma_phi.alpha.L} that $L=h_{\nabla}$ (since $\alpha_L=\nabla)$, which gives $h_{\nabla}+h_{\nabla}^\ast=dh$, showing that $\nabla$ is compatible with $h$.
\end{proof}

\noindent
Let $\nabla^0$ be an arbitrary connection on $M$ and assume that $h$ is an invertible hermitian form on $M$. Setting
\begin{align*}
    L = \tfrac{1}{2}dh + \tfrac{1}{2}(h_{\nabla^0}-h_{\nabla^0}^\ast) + A,
\end{align*}
for arbitrary $A\in\Hom_{\A}(M\times\Mh,\Omegabg{1})$ such that $A^\ast=-A$ gives $L+L^\ast=dh$ and
\begin{align*}
    &L(am_1,m_2) = aL(m_1,m_2) + (da)h(m_1,m_2)\\
    &L(m_1,am_2) = L(m_1,m_2)a^\ast
\end{align*}
implying that $L\in\Hom_{\complex,\A}(M\times\Mh,\Omegabg{1})$. Using Proposition~\ref{prop:L.hhi.connection} one obtains the following result.

\begin{proposition}\label{prop:existence.metric.connection}
    Let $(\A,\g)$ be a derivation based calculus and let $M$ be an $\A$-module. Moreover, let $h$ be an invertible hermitian form on $M$ and let $\nabla^0\in\Cg(M)$. Then  \begin{align*}
        &\phi_{\nabla^0}:\{A\in\Hom_{\A}(M\times\Mh,\Omegaonebg):A^\ast=-A\}\to\Cg(M,h)\\
        &\phi_{\nabla^0}(A)(m,\d) = \hhi(\phi^{L_A}_{m,\d})
    \end{align*}
    for $m\in M$ and $\d\in\g$, where    
    \begin{align}\label{eq:def.L}
        L_A = \tfrac{1}{2}dh + \tfrac{1}{2}(h_{\nabla^0}-h_{\nabla^0}^\ast) + A,
    \end{align}
    is a bijection.
\end{proposition}

\begin{proof}
    Assume that $\nabla^0\in\Cg(M)$. Let us first show that for any $A\in\Hom_{\A}(M\times\Mh,\Omegaonebg)$ such that $A^\ast=-A$, $\phi_{\nabla^0}(A)$ is a connection compatible with $h$. To start with, we shall use Proposition~\ref{prop:L.hhi.connection} to show that $\phi_{\nabla^0}(A)$ is a connection, with $L_A$ defined as in \eqref{eq:def.L}. From
    \begin{align*}
        L_A(m_1,m_2)(\d) = \thalf\d h(m_1,m_2) 
        + \thalf h(\nabla^0_{\d}m_1,m_2)-\thalf h(m_1,\nabla^0_{\d^\ast}m_2)
        +A(m_1,m_2)(\d)
    \end{align*}
    it is clear that $L_A$ is biadditive and $L_A(\lambda m_1,m_2)(\d)=\lambda L_A(m_1,m_2)(\d)$ for $\lambda\in\complex$, as well as 
    \begin{align*}
        L_A(m_1,m_2)(z_1 \d_1+z_2 \d_2) = z_1 L_A(m_1,m_2)(\d_1)+z_2 L_A(m_1,m_2)(\d_2)   
    \end{align*}
    for $z_1,z_2\in\ZA$ such that $z_1\d_1,z_2\d_2\in\g$, showing that $L_A(m_1,m_2)\in\Omegaonebg$. Moreover,
    \begin{align*}
        L_A(m_1,am_2)(\d)
        &=\thalf h(m_1,m_2)\d a^\ast + \thalf \paraa{\d h(m_1,m_2)}a^\ast
        +\thalf h(\nabla^0_{\d}m_1.m_2)a^\ast\\
        &-\thalf h(m_1,\nabla^0_{\d^\ast}m_2)a^\ast-\thalf h(m_1,m_2)(\d^\ast a)^\ast
        +A(m_1,m_2)a^\ast\\
        &= L_A(m_1,m_2)(\d)a^\ast,
    \end{align*}
    showing that $L_A\in\Hom_{\complex,\A}(M\times\Mh,\Omegaonebg)$ and, furthermore, one finds that
    \begin{align*}
        L_A(am_1,m_2)(\d) 
        &= \thalf(\d a)h(m_1,m_2)+\thalf a \d h(m_1,m_2)
        +\thalf(\d a)h(m_1,m_2)\\
        &\qquad+\thalf ah(\nabla^0_{\d}m_1,m_2)-\thalf ah(m_1,\nabla^0_{\d^\ast}m_2)
        +aA(m_1,m_2)(\d)\\
        &=aL_A(m_1,m_2)(\d) + (\d a)h(m_1,m_2),
    \end{align*}
    and it follows from Proposition~\ref{prop:L.hhi.connection} that $\phi_{\nabla^0}(A)=\hhi(\phi^L_{m,\d})$ defines a connection on $M$. Using that $A^\ast=-A$ and $dh^\ast=dh$ one concludes that $L_A+L_A^\ast=dh$ which implies that $\nabla$ is compatible with $h$, by Proposition~\ref{prop:L.hhi.connection}. Furthermore, $\phi_{\nabla^0}$ is injective since
    \begin{align*}
        &\phi_{\nabla^0}(A) = \phi_{\nabla^0}(A')\implies
        \phi^{L_A}_{m,\d}(m') = \phi^{L_{A'}}_{m,\d}(m')\implies\\
        &L_A(m,m')(\d)=L_{A'}(m,m')(\d)\implies
        A(m,m')(\d) = A'(m,m')(\d)
    \end{align*}
    for all $\d\in\g$ and $m,m'\in M$, which is equivalent to $A=A'$.

    Now, let us show that $\phi_{\nabla^0}$ is surjective; to this end, let $\nabla\in\Cg(M,h)$. Define
    \begin{align*}
        A(m_1,m_2)(\d) 
        &= h(\nabla_{\d}m_1,m_2) - \thalf \d h(m_1,m_2)
        -\thalf h(\nabla^0_{\d}m_1,m_2)+\thalf h(m_1,\nabla^0_{\d^\ast}m_2),
    \end{align*}
    from which it is clear that $A(m_1,m_2)\in\Omegaonebg$. Moreover, one checks that
    \begin{align*}
        A(am_1,bm_2)(\d)
        &= (\d a)h(m_1,m_2)b^\ast+ah(\nabla_{\d}m_1,m_2)b^\ast
        -\thalf (\d a)h(m_1,m_2)b^\ast\\
        &\quad -\thalf a(\d h(m_1,m_2))b^\ast
        -\thalf ah(m_1,m_2)(\d b^\ast)
        -\thalf a h(\nabla^0_{\d}m_1,m_2)b^\ast \\
        &\quad- \thalf(\d a)h(m_1,m_2)b^\ast
        +\thalf ah(m_1,m_2)(\d^\ast b)^\ast + \thalf ah(m_1,\nabla^0_{\d}m_2)b^\ast\\
        &=(aA(m_1,m_2)b^\ast)(\d)
    \end{align*}
    which, together with biadditivity, shows that $A\in\Hom_{\A}(M\times\Mh,\Omegaonebg)$. Furthermore, one finds that
    \begin{align*}
        A^\ast&(m_1,m_2)(\d) 
        = A(m_2,m_1)^\ast(\d) = A(m_2,m_1)(\d^\ast)^\ast\\
        &= \paraa{h(\nabla_{\d^\ast}m_2,m_1) - \thalf \d^\ast h(m_2,m_1)
        -\thalf h(\nabla^0_{\d^\ast}m_2,m_1)+\thalf h(m_2,\nabla^0_{\d}m_1)}^\ast\\
        &=h(m_1,\nabla_{\d^\ast}m_2) - \thalf \d h(m_1,m_2)
        -\thalf h(m_1,\nabla^0_{\d^\ast}m_2)+\thalf h(\nabla^0_{\d}m_1,m_2),
    \end{align*}
    and using that $\nabla$ is compatible with $h$, implying that
    \begin{align*}
        h(m_1,\nabla_{\d^\ast}m_2)= \d h(m_1,m_2)-h(\nabla_{\d}m_1,m_2),
    \end{align*}
    gives
    \begin{align*}
        A^\ast(m_1,m_2)(\d)
        &= -h(\nabla_{\d}m_1,m_2) + \thalf \d h(m_1,m_2)
        -\thalf h(m_1,\nabla^0_{\d^\ast}m_2)+\thalf h(\nabla^0_{\d}m_1,m_2)\\
        &=-A(m_1,m_2)(\d).
    \end{align*}
    We conclude that there exists $A\in\Hom_{\A}(M\times\Mh,\Omegaonebg)$ such that $A^\ast=-A$ and $\phi_{\nabla^0}(A)=\nabla$. Hence, $\phi_{\nabla^0}$ is surjective.
\end{proof}

\noindent
As an immediate consequence, one obtains the following corollary.

\begin{corollary}\label{cor:existence.metric.connections}
    If $(\A,\g,h)$ is a hermitian calculus then there exists a connection on $\Omegaoneg$ compatible with $h$.
\end{corollary}

\begin{proof}
    If $(\A,\g,h)$ is a hermitian calculus then there exists a connection $\nabla^0$ on $\Omegaoneg$. Since the metric is invertible, Proposition~\ref{prop:existence.metric.connection} implies that $\phi_{\nabla^0}(0)$ is a connection on $\Omegaoneg$ compatiable with $h$.
\end{proof}

\noindent To obtain a better understanding of the compatible connection defined in Proposition~\ref{prop:existence.metric.connection}, let us work out the case when $M$ is a projective module with generators $\{\theta^i\}_{i=1}^N$, and dual basis $\{\phi_i\}_{i=1}^N$, where $\nabla^0$ is the so called Grassmann connection. 

As noted in Section~\ref{sec:ind.cal.proj.mods}, given an invertible hermitian form $h$ one can realize the projective module $M$ as the projection of the free module $\A^N$ with basis $\{e^i\}_{i=1}^N$ with respect to
\begin{align*}
    p(m_ie^i) = m_ih^{ij}h_{jk}e^k,
\end{align*}
giving $p(\A^N)\simeq M$ and the isomorphism is given by $\psi=\pi|_{p(\A^N)}$. As generators of $p(\A^N)$ we introduce $\thetat^i= h^{ij}h_{jk}e^k$ and note that $\psi(\thetat^i)=\theta^i$.

If one defines a connection $\nablat^0$ on $\A^N$ as 
\begin{align*}
    \nablat^0_{\d}(m_ie^i) = (\d m_i)e^i
\end{align*}
then $p\circ\nablat^0$ is a connection on $p(\A^N)$ and, consequently,
\begin{align*}
    \nabla^0 = \psi\circ p \circ \nablat^0\circ\psi^{-1}
\end{align*}
is a connection on $M$. One finds that
\begin{equation}\label{eq:grassm.connection}
    \begin{split}
        \nabla^0_\d\theta^i &= (\psi\circ p)\paraa{\nablat^0_\d\thetat^i}
        = (\psi\circ p)\paraa{(\d h^{ij}h_{jk})e^k}
        = (\d h^{ij}h_{jk})\psi(\thetat^k)\\
        &=(\d h^{ij}h_{jk})\theta^k.
    \end{split}
\end{equation}
Now, set 
\begin{align*}
    L = \thalf dh + \thalf h_{\nabla^0} - \thalf h_{\nabla^0}^\ast
\end{align*}
and $\phi^L_{m,\d}(m') = L(m,m')(\d)^\ast$, 
giving the connection in Proposition~\ref{prop:existence.metric.connection} as
\begin{align*}
    \nabla_{\d} m = \hhi(\phi^L_{m,\d}) 
    = \phi^L_{m,\d}(\theta^i)^\ast\theta_i 
    = L(m,\theta^i)(\d)\theta_i.
\end{align*}
For the Grassmann connection $\nabla^0$ given in \eqref{eq:grassm.connection} one obtains
\begin{align*}
    \nabla_{\d}\theta^i 
    &= \thalf(\d h^{ij})\theta_j + \thalf h(\nabla^0_{\d}\theta^i,\theta^j)\theta_j
    -\thalf h(\theta^i,\nabla^0_{\d^\ast}\theta^j)\theta_j\\
    &= \thalf(\d h^{ij})\theta_j + \thalf (\d h^{ik}h_{kl})h^{lj}\theta_j
    -\thalf h^{il}(\d h_{lk}h^{kj})\theta_j\\
    &=\thalf(\d h^{ij})\theta_j + \thalf (\d h^{ik}h_{kl})h^{lj}\theta_j
    -\thalf(\d h^{il}h_{lk}h^{kj})\theta_j+\thalf (\d h^{il})\theta_l
\end{align*}
giving
\begin{align}
    \nabla_{\d}\theta^i = 
    \thalf(\d h^{ij})\theta_j + \thalf (\d h^{ij}h_{jk})\theta^k,
\end{align}
as a connection on $M$ compatible with $h$.

\section{Torsion free connections on $\Omegaoneg$}\label{sec:torsionfree.conn}

\noindent 
After having considered the existence of connections compatible with a hermitian form in the previous section, let us turn our attention to torsion free connections. 
In differential geometry, the torsion $T$ of a connection $\nabla$ on the tangent bundle is defined as
\begin{align*}
    T(\nabla)(X,Y) = \nabla_X Y-\nabla_Y X-[X,Y]
\end{align*}
for vector fields $X,Y$ on the manifold. In terms of the corresponding connection on differential forms, one obtains
\begin{align*}
    (T(\nabla)\omega)(X,Y) = (\nabla_X\omega)(Y)-(\nabla_Y\omega)(X)-d\omega(X,Y),
\end{align*}
and a connection is called torsion free if the above expressions vanish.

In general, there is no canonical way of defining torsion for connections on arbitrary modules (however, see \cite{aw:curvature.three.sphere,aw:cgb.sphere,atn:minimal.embeddings.morphisms} for a formulation of torsion in terms of anchor maps), but for $\Omegaoneg$ there is a straightforward generalization of the classical concept. 

\begin{definition}
    Let $(\A,\g)$ be a derivation based calculus and let $\nabla\in\Cg(\Omegaoneg)$. The torsion $T(\nabla):\Omegaoneg\to\Omegabg{2}$ is defined as
    \begin{align}\label{eq:def.torsion}
        \paraa{T(\nabla)\omega}(\d_1,\d_2) =
        (\nabla_{\d_1}\omega)(\d_2)-(\nabla_{\d_2}\omega)(\d_1)-d\omega(\d_1,\d_2) 
    \end{align}
    for $\omega\in\Omegaoneg$ and $\d_1,\d_2\in\g$. Moreover, a connection is called torsion free if $T(\nabla)=0$. The set of torsion free connections on $\Omegaoneg$ will be denoted by $\CgT(\Omegaoneg)$.
\end{definition}

\noindent
Given a connection $\nabla$ on $\Omegaoneg$, one can define 
\begin{align}\label{eq:associated.connection}
    (\nablat_{\d}\omega)(\d') = (\nabla_{\d'} \omega)(\d) + d\omega(\d,\d').
\end{align}
and it is easy to check that $\nablat$ satisfies the properties required for a connection.
However, even though $\omega\in\Omegaoneg$, it is not guaranteed that $\nablat_{\d}\omega$ is an element of $\Omegaoneg$ (but clearly an element of $\Omegaonebg$). Moreover, since $d\omega\in\Omegag{2}$ for $\omega\in\Omegaoneg$, it follows that $d\omega(\d,\cdot)$ is an element of $\Omegaoneg$, so the potential problem lies with 
$(\nabla_{\d'}\omega)(\d)$ for fixed $\d\in\g$ and $\omega\in\Omegaoneg$.

If $\nabla$ is torsion free, then $\nabla=\nablat$ (which is simply a rearrangement of \eqref{eq:def.torsion} when $T(\nabla)=0)$ implying that $\nablat_{\d}\omega=\nabla_{\d}\omega\in\Omegaoneg$. Hence, the requirement that $\nablat$ is a connection on $\Omegaoneg$ is a necessary condition for torsion freeness. To address this issue we introduce the following definition.

\begin{definition}\label{def:regular.map}
    Let $\alpha\in\Hom_{\complex,\ZhA}(\Omegaoneg\times\g,\Omegaoneg)$ and define 
    \begin{align*} 
        &\sigma(\alpha)\in\Hom_{\complex,\ZhA}(\Omegaoneg\times\g,\Omegaonebg)\\
        &\sigma(\alpha)(\omega,\d)(\d') = \alpha(\omega,\d')(\d).
    \end{align*}
    If $\sigma(\alpha)\in\Hom_{\complex,\ZhA}(\Omegaoneg\times\g,\Omegaoneg)$ then $\alpha$ is called regular. The set of regular maps in $\Hom_{\complex,\ZhA}(\Omegaoneg\times\g,\Omegaoneg)$ will be denoted by $\HomReg_{\complex,\ZhA}(\Omegaoneg\times\g,\Omegaoneg)$.
\end{definition}

\noindent
The set of regular connections on $\Omegaoneg$ will be denoted by $\CgReg(\Omegaoneg)$. If $\nabla\in\CgReg(\Omegaoneg)$ then it is clear that $\nablat$, defined by \eqref{eq:associated.connection}, is also a connection on $\Omegaoneg$. 

In view of the discussion above, a necessary condition for the existence of torsion free connections is the existence of regular connections, i.e. $\CgT(\Omegaoneg)\subseteq\CgReg(\Omegaoneg)$. Let us illustrate the concept of regular connections in the case of free modules.
 
\begin{example}
    Assume that $\Omegaoneg$ is a free module with basis $\{\theta^i\}_{i=1}^N$. A connection on $\Omegaoneg$ is given by choosing $\{\Gamma^i_j\}_{i,j=1}^N\subseteq\Omegaonebg$ and setting $\nabla_{\d}\theta^i=\Gamma^i_j(\d)\theta^j$ giving
    \begin{align*}
        \nabla_{\d}(\omega_i\theta^i)
        =\omega_i\Gamma^i_j(\d)\theta^j + (\d\omega_i)\theta^i
    \end{align*}
    which is clearly an element of $\Omegaoneg$. Furthermore
    \begin{align*}
        \sigma(\nabla)(\omega,\d)(\d') = \Gamma^i_j(\d')\theta^j(\d), 
    \end{align*}
    and we conclude that $\nabla$ is a regular connection if and only if $\Gamma^i_j\in\Omegaoneg$, giving a concrete characterization of regular connections in this case.
\end{example}

\begin{remark}
    Note that, in differential geometry, given a connection $\nabla$ on the tangent bundle, the associated connection in \eqref{eq:associated.connection} corresponds to
    \begin{align*}
        \nablat_{X} Y = \nabla_Y X + [X,Y],
    \end{align*}
    for vector fields $X,Y$.
\end{remark}
 
\noindent
Let us introduce some notation in order to rewrite \eqref{eq:def.torsion} in a more convenient form. To this end, define
\begin{align*}
    \wedge,s:\Hom_{\complex,\ZhA}(\Omegaoneg\times\g,\Omegabg{1})
    \to\Hom_{\complex,\ZhA}(\Omegaoneg\times\g,\Omegabg{1})
\end{align*}
as
\begin{align*}
    &\wedge\alpha = \alpha-\sigma(\alpha)\qand s(\alpha) = \alpha + \sigma(\alpha)
\end{align*}
that is,
\begin{align*}
    &(\wedge\alpha)(\omega,\d)(\d') = \alpha(\omega,\d){\d'} - \alpha(\omega,\d')(\d)\\
    &s(\alpha)(\omega,\d)(\d') = \alpha(\omega,\d){\d'} + \alpha(\omega,\d')(\d)
\end{align*}
from which it follows that $s\circ\wedge=\wedge\circ s = 0$. Note that if $\alpha\in\HomReg_{\complex,\ZhA}(\Omegaoneg\times\g,\Omegaoneg)$ then $\wedge \alpha,s(\alpha)\in\Hom_{\complex,\ZhA}(\Omegaoneg\times\g,\Omegaoneg)$. 
In this context, we allow for a slight abuse of notation and consider the exterior derivative $d\in\HomReg_{\complex,\ZhA}(\Omegaoneg\times\g,\Omegaoneg)$ via
\begin{align*}
    d(\omega,\d_1)(\d_2) = d\omega(\d_1,\d_2)
\end{align*}
satisfying $\wedge d = 2d$ and $s(d)=0$. 
Hence, one can write \eqref{eq:def.torsion} as
\begin{align}
    T(\nabla) = \wedge\nabla - d. 
\end{align}
As we have seen, the existence of regular connections is necessary for the existence of torsion free connections, and the next result gives an explicit characterization of all torsion free connections on $\Omegaoneg$.

\begin{proposition}\label{prop:torsionfree.connections}
    Let $(\A,\g)$ be a derivation based calculus and let $\nabla^0$ be a regular left $\g$-connection on $\Omegaoneg$. Then
    \begin{align*}
        &\psi_{\nabla^0}:\{\beta\in\HomReg_{\A,\ZhA}(\Omegaoneg\times\g,\Omegaoneg):\wedge\beta=0\}\to\CgT(\Omegaoneg)\\  &\psi_{\nabla^0}(\beta) = \tfrac{1}{2}\paraa{d + s(\nabla^0)} + \beta
    \end{align*}
    is a bijection.
\end{proposition}

\begin{proof}
    Let $\nabla^0\in\CgReg(\Omegaoneg)$ and assume that $\beta\in\HomReg_{\A,\ZhA} (\Omegaoneg\times\g,\Omegaoneg)$ such that $\wedge\beta=0$. Defining
    \begin{align*}
        \nabla = \psi_{\nabla^0}(\beta) = \tfrac{1}{2}\paraa{d + s(\nabla^0)} + \beta
    \end{align*}
    it follows that $\nabla\in\Hom_{\complex,\ZhA}(\Omegaoneg\times\g,\Omegaoneg)$ since $\nabla^0$ is assumed to be regular. Furthermore, one checks that
    \begin{align*}
        (\nabla_{\d}a\omega)(\d') 
        &= \thalf d(a\omega)(\d,\d')
        +\thalf (\nabla^0_{\d} a\omega)(\d')
        +\thalf (\nabla^0_{\d'} a\omega)(\d)+\beta(a\omega,\d)(\d')\\
        &= a(\nabla_{\d}\omega)(\d')
        +\thalf (da\cdot\omega)(\d,\d')
        +\thalf (\d a)\omega(\d')
        +\thalf (\d'a)\omega(\d)\\
        &=a(\nabla_{\d}\omega)(\d') + (\d a)\omega(\d'),
    \end{align*}
    as well as $\nabla_{z\d}\omega=z\nabla_{\d}\omega$ for $\omega\in\Omegaoneg$, $z\in\ZA$ and $\d\in\g$ such that $z\d\in\g$. Together with the obvious linearity properties this shows that $\nabla\in\Cg(\Omegaoneg)$. Using that $\wedge d = 2d$ and $\wedge\circ s=0$, together with the assumption $\wedge\beta=0$, one finds that
    \begin{align*}
        \wedge\nabla = \thalf\wedge d = d
    \end{align*}
    showing that $\nabla=\psi_{\nabla^0}(\beta)\in\CgT(\Omegaoneg)$. The injectivity of $\psi_{\nabla^0}$ is clear since
    \begin{align*}
        \psi_{\nabla^0}(\beta) = \psi_{\nabla^0}(\beta')\implies
        \thalf(d+s(\nabla^0)) + \beta = \thalf(d+s(\nabla^0))+\beta'\implies
        \beta=\beta'.
    \end{align*}
    Let us now show that $\psi_{\nabla^0}$ is surjective. To this end, assume that $\nabla\in\CgT(\Omegaoneg)$ and $\nabla^0\in\CgReg(\Omegaoneg)$. Defining
    \begin{align*}
        \beta = \nabla-\thalf\paraa{d+s(\nabla^0)}
    \end{align*}
    implies that $\nabla=\psi_{\nabla^0}(\beta)$ and $\beta\in\HomReg_{\complex,\ZhA}(\Omegaoneg\times\g,\Omegaoneg)$ since $\nabla,\nabla^0$ are regular. Moreover,
    \begin{align*}
        \wedge\beta = \wedge\nabla - d = 0,
    \end{align*}
    since $\nabla$ is torsion free, and
    \begin{align*}
        \beta(a\omega,\d)(\d')
        &= (\nabla_{\d}a\omega)(\d')
        -\tfrac{1}{2}d(a\omega)(\d,\d') 
        - \tfrac{1}{2}(\nabla^0_{\d}a\omega)(\d')
        - \tfrac{1}{2}(\nabla^0_{\d'}a\omega)(\d)\\
        &=a(\nabla_{\d}\omega)(\d')+(\d a)\omega(\d')
        -\tfrac{1}{2}(\d a)\omega(\d')+\tfrac{1}{2}(\d'a)\omega(\d)
        -\tfrac{1}{2}ad\omega(\d,\d')\\
        &\qquad-\tfrac{1}{2}a(\nabla^0_{\d}\omega)(\d')-\tfrac{1}{2}(\d a)\omega(\d')-\tfrac{1}{2}a(\nabla^0_{\d'}\omega)(\d)-\tfrac{1}{2}(\d' a)\omega(\d)\\
        &=a\beta(\omega,\d)(\d')
    \end{align*}
    shows that $\beta\in\HomReg_{\A,\ZhA}(\Omegaoneg\times\g,\Omegaoneg)$. 
    Hence, $\psi_{\nabla^0}$ is surjective. 
\end{proof}

\noindent
In particular, one obtains the following result.
\begin{corollary}\label{cor:existence.torsion.free.connection}
    Let $(\A,\g)$ be a derivation based calculus. There exists a torsion free connection on $\Omegaoneg$ if and only if $\CgReg(\Omegaoneg)\neq\emptyset$.
\end{corollary}

\begin{proof}
    First, assume that $\CgReg(\Omegaoneg)\neq\emptyset$ and let $\nabla^0\in\CgReg(\Omegaoneg)$. Clearly, $\beta=0$ satisfies the requirements that $\beta\in\Hom_{\A,\ZhA}(\Omegaoneg\times\g,\Omegaoneg)$ and $\wedge\beta=0$. It follows from Proposition~\ref{prop:torsionfree.connections} that $\psi_{\nabla^0}(0)$ is a torsion free connection on $\Omegaoneg$.
    Conversely, let $\nabla\in\CgT(\Omegaoneg)$. Since $\nabla\in\CgT(\Omegaoneg)\subseteq\CgReg(\Omegaoneg)$ it follows that $\CgReg(\Omegaoneg)\neq\emptyset$.
\end{proof}

\section{Levi-Civita connections}\label{sec:lc.connections}

\noindent
In the previous sections we have separately investigated the existence of torsion free connections and connections compatible with a hermitian form. We recall from Corollary~\ref{cor:existence.metric.connections} and Corollary~\ref{cor:existence.torsion.free.connection} that torsion free connections, as well as connections compatible with the hermitian form, exist on hermitian calculi (assuming $\CgReg(\Omegaoneg)\neq\emptyset$). In this   section, we combine these results and explore the existence of Levi-Civita connections on $\Omegaoneg$. 

\begin{definition}
    Let $(\A,\g)$ be a derivation based calculus and let $h$ be a hermitian form on $\Omegaoneg$. A connection $\nabla\in\Cg(\Omegaoneg)$ is called a Levi-Civita connection with respect to $h$ if it is torsion free and compatible with $h$. The set of Levi-Civita connections with respect to $h$ will be denoted by $\CgLC(\Omegaoneg,h)$.
\end{definition}

\noindent
Comparing with Proposition~\ref{prop:existence.metric.connection} and Proposition~\ref{prop:torsionfree.connections}, we would like to describe the set of Levi-Civita connections as a bijective correspondence with a set of maps that depend on the hermitian form. 

To this end, for $\alpha\in\Hom_{\complex,\ZhA}(\Omegaoneg\times\g,\Omegaoneg)$ define $T_h(\alpha)\in\Hom_{\complex}(\Omegaoneg\times\Omegaonegh,\Omegaoneg)$ by
\begin{align*}
    T_h(\alpha) = h_\alpha + h_\alpha^\ast,
\end{align*}
which enables one to write the compatibility condition with $h$ as $dh=T_h(\nabla)$.

The next result characterizes the set of Levi-Civita connections in terms of the inverse image of the element $2dh-T_h(d)\in\Hom_{\complex}(\Omegaoneg\times\Omegaonegh,\Omegaonebg)$. In Riemannian geometry, this set is not of particular interest since it contains only one element, due to the fact that the Levi-Civita connection is unique. However, as we shall see explicitly in Section~\ref{sec:lc.free.modules}, this need not be the case in the context of derivation based calculi. In any case, one has the following result, where one should note that the hermitian form is not assumed to be invertible.

\begin{theorem}\label{thm:LC.inverse.image}
    Let $(\A,\g)$ be a derivation based calculus and let $h$ be a hermitian form on $\Omegaoneg$. Then
    \begin{align*}
        &\Psi:\paraa{(T_h\circ s)|_{\CgReg(\Omegaoneg)}}^{-1}\paraa{2dh-T_h(d)}\to \CgLC(\Omegaoneg,h)\\
        &\Psi(\nabla) = \tfrac{1}{2}\paraa{d+s(\nabla)}
    \end{align*}
    is bijective.
\end{theorem}

\begin{proof}
    Let $\nabla\in\paraa{(T_h\circ s)|_{\CgReg(\Omegaoneg)}}^{-1}\paraa{2dh-T_h(d)}$. Proposition~\ref{prop:torsionfree.connections} (with $\beta=0$) implies that $\Psi(\nabla)$  is a torsion free connection on $\Omegaoneg$. Moreover,
    \begin{align*}
        h_{\Psi(\nabla)} + h_{\Psi(\nabla)}^\ast = T_h(\Psi(\nabla)) = \thalf T_h(d) + \thalf (T_h\circ s)(\nabla)
        =\thalf T_h(d)+dh-\thalf T_h(d) = dh
    \end{align*}
    since $\nabla\in\paraa{(T_h\circ s)|_{\CgReg(\Omegaoneg)}}^{-1}\paraa{2dh-T_h(d)}$. Thus, $\Psi(\nabla)$ is compatible with $h$ which, together with torsion freeness, implies that $\Psi(\nabla)\in\CgLC(\Omegaoneg)$. To prove surjectivity we note that if $\nabla\in\CgLC(\Omegaoneg)$ then $\nabla\in\CgReg(\Omegaoneg)$ and
    \begin{align*}
        \Psi(\nabla) = \thalf d + \thalf s(\nabla) = \thalf d+\nabla-\thalf\wedge\nabla
        =\thalf d +\nabla - \thalf d = \nabla
    \end{align*}
    since $\nabla$ is torsion free. Using that $\Psi(\nabla)=\nabla$ together with the fact that $\nabla$ is compatible with $h$, one obtains
    \begin{align*}
        &dh = h_\nabla + h_\nabla^\ast = T_h(\nabla)
        = \thalf T_h(d) + \thalf(T_h\circ s)(\nabla)\implies\\
        &(T_h\circ s)(\nabla) = 2dh-T_h(d).
    \end{align*}
    Hence, $\nabla\in\paraa{(T_h\circ s)|_{\CgReg(\Omegaoneg)}}^{-1}\paraa{2dh-T_h(d)}$ and $\Psi(\nabla)=\nabla$ which implies that $\Psi$ is surjective. Let us now show that $\Psi$ is injective. 
    
    Assume that $\nabla,\nabla'\in\paraa{(T_h\circ s)|_{\CgReg(\Omegaoneg)}}^{-1}\paraa{2dh-T_h(d)}$ and $\Psi(\nabla) = \Psi(\nabla')$, which implies that
    \begin{align*}
        \thalf d + s(\nabla) = \thalf d + s(\nabla')\equivalent s(\nabla-\nabla') = 0.
    \end{align*}
    Next, note that if $s(\beta)=0$ then $\sigma(\beta)=-\beta$ and, consequently, $\wedge(\tfrac{1}{2}\beta)=\tfrac{1}{2}(\beta-\sigma(\beta))=\beta$. Hence, since $s(\nabla-\nabla')=0$, there exists $\alpha\in\Hom_{\A,\ZhA}(\Omegaoneg\times\g,\Omegaoneg)$ such that $\nabla'=\nabla+\wedge\alpha$. Moreover, since $\nabla,\nabla'$ are torsion free, one obtains
    \begin{align*}
        d = \wedge\nabla' = \wedge\nabla + 2(\wedge\alpha) = d+2(\wedge\alpha)\implies
        \wedge\alpha = 0,
    \end{align*}
    which implies that $\nabla'=\nabla$. Hence, $\Psi$ is injective.
\end{proof}

\noindent
In particular, the above results gives an equivalent condition for the existence of Levi-Civita connections, which we state as follows:
\begin{corollary}\label{cor:LC,equiv.existence}
    Let $(\A,\g)$ be a derivation based calculus and let $h$ be a hermitian form on $\Omegaoneg$. There exists a Levi-Civita connection on $\Omegaoneg$ if and only if there exists $\nabla\in\CgReg(\Omegaoneg)$ such that 
    \begin{align*}
        (T_h\circ s)(\nabla) = 2dh-T_h(d).
    \end{align*}
\end{corollary}

\noindent
The results in Theorem~\ref{thm:LC.inverse.image} and Corollary \ref{cor:LC,equiv.existence} characterize the set of Levi-Civita connections and provide a necessary and sufficient condition for their existence. Let us now derive a, perhaps more conceptual, necessary condition for the hermitian form. Namely, as already mentioned in Section~\ref{sec:gen.metric.symmetry}, the hermitian form has to be weakly symmetric in order for Levi-Civita connections exist, as shown in the following result.  

\begin{proposition}\label{prop:LC.necessary.cond}
    Let $(\A,\g,h)$ be a hermitian calculus. If there exists a Levi-Civita connection with respect to $h$ on $\Omegaoneg$ then $(\A,\g,h)$ is weakly symmetric.
\end{proposition}

\begin{proof}
    Assume that $(\A,\g,h)$ is a hermitian calculus, and let $\nabla$ be a Levi-Civita connection with respect to $h$. To show that $(\A,\g,h)$ is weakly symmetric, we need to show that $d\rho=0$. Define $\vh:\g\to\Omegaoneg$ as
    \begin{align*}
        \vh(\d) = \hhi\paraa{\varphi(\d^\ast)}
    \end{align*}
    satsifying $\vh(z\d)=z\vh(\d)$ for $z\in\ZA$ and $\d\in\g$ such that $z\d\in\g$. In this notation, the symmetry form may be written as
    \begin{align*}
        \rho(\d_1,\d_2) &= \hi\paraa{\varphi(\d_1^\ast),\varphi(\d_2)}-\hi\paraa{\varphi(\d_2^\ast),\varphi(\d_1)}\\
        &=h\paraa{(\hhi\circ\varphi)(\d_1^\ast),(\hhi\circ\varphi)(\d_2)}
        -h\paraa{(\hhi\circ\varphi)(\d_2^\ast),(\hhi\circ\varphi)(\d_1)}\\
        &=h\paraa{\vh(\d_1),\vh(\d_2^\ast)}-h\paraa{\vh(\d_2),\vh(\d_1^\ast)}.
    \end{align*}
    Next, one notes that
    \begin{align*}
        h\paraa{\omega,\vh(\d^\ast)} = h\paraa{\omega,\hhi(\varphi(\d))} = \varphi(\d)(\omega)=\omega(\d),
    \end{align*}
    and, due to the compatibility of $\nabla$ with $h$, it follows that
    \begin{equation}\label{eq:dh.rewrite}
    \begin{split}
        \d h\paraa{\vh(\d_1),\vh(\d_2^\ast)} 
        &= h\paraa{\nabla_{\d}\vh(\d_1),\vh(\d_2^\ast)}+h\paraa{\vh(\d_1),\nabla_{\d^\ast}\vh(\d_2^\ast)}\\
        &= \paraa{\nabla_{\d}\vh(\d_1)}(\d_2)+\paraa{\nabla_{\d^\ast}\vh(\d_2^\ast)}(\d_1^\ast)^\ast.        
    \end{split}
    \end{equation}
    Now, the expression for $d\rho$ is
    \begin{align*}
        d\rho(\d_1,\d_2,\d_3) 
        &= \d_1\rho(\d_2,\d_3) + \d_2\rho(\d_3,\d_1)+\d_3\rho(\d_1,\d_2)\\
        &\qquad-\rho([\d_1,\d_2],\d_3)-\rho([\d_2,\d_3],\d_1)-\rho([\d_3,\d_1],\d_2)
    \end{align*}
    and using \eqref{eq:dh.rewrite} to rewrite the first three terms gives
    \begin{align*}
        \d_1&\rho(\d_2,\d_3) + \d_2\rho(\d_3,\d_1)+\d_3\rho(\d_1,\d_2)\\
        &=\paraa{\nabla_{\d_1}\vh(\d_2)}(\d_3)+\paraa{\nabla_{\d_1^\ast}\vh(\d_3^\ast)}(\d_2^\ast)^\ast
        -\paraa{\nabla_{\d_1}\vh(\d_3)}(\d_2)-\paraa{\nabla_{\d_1^\ast}\vh(\d_2^\ast)}(\d_3^\ast)^\ast\\
        &+\paraa{\nabla_{\d_2}\vh(\d_3)}(\d_1)+\paraa{\nabla_{\d_2^\ast}\vh(\d_1^\ast)}(\d_3^\ast)^\ast
        -\paraa{\nabla_{\d_2}\vh(\d_1)}(\d_3)-\paraa{\nabla_{\d_2^\ast}\vh(\d_3^\ast)}(\d_1^\ast)^\ast\\
        &+\paraa{\nabla_{\d_3}\vh(\d_1)}(\d_2)+\paraa{\nabla_{\d_3^\ast}\vh(\d_2^\ast)}(\d_1^\ast)^\ast
        -\paraa{\nabla_{\d_3}\vh(\d_2)}(\d_1)-\paraa{\nabla_{\d_3^\ast}\vh(\d_1^\ast)}(\d_2^\ast)^\ast,
    \end{align*}
    and, furthermore, since $\nabla$ is torsion free one obtains
    \begin{align*}
        \d_1&\rho(\d_2,\d_3) + \d_2\rho(\d_3,\d_1)+\d_3\rho(\d_1,\d_2)\\        &=d(\vh(\d_1^\ast))^\ast(\d_2,\d_3)+d(\vh(\d_2^\ast))^\ast(\d_3,\d_1)+d(\vh(\d_3^\ast))^\ast(\d_1,\d_2)\\
        &\qquad-d(\vh(\d_1))(\d_2,\d_3)-d(\vh(\d_2))(\d_3,\d_1)-d(\vh(\d_3))(\d_1,\d_2)\\
        &= d\paraa{\vh(\d_1^\ast)^\ast-\vh(\d_1)}(\d_2,\d_3)
        +d\paraa{\vh(\d_2^\ast)^\ast-\vh(\d_2)}(\d_3,\d_1)\\
        &\qquad+d\paraa{\vh(\d_3^\ast)^\ast-\vh(\d_3)}(\d_1,\d_2).
    \end{align*}
    Noting that $\vh(\d_1)(\d_2)=h\paraa{\vh(\d_1),\vh(\d_2^\ast)}$, one finds
    \begin{align*}
        \paraa{\vh(\d_1^\ast)^\ast-\vh(\d_1)}(\d_2) = -\rho(\d_1,\d_2)
    \end{align*}
    giving
    \begin{align*}
        d\paraa{\vh(\d_1^\ast)^\ast-\vh(\d_1)}(\d_2,\d_3)
        &=\d_2\rho(\d_3,\d_1)+\d_3\rho(\d_1,\d_2)+\rho(\d_1,[\d_2,\d_3])
    \end{align*}
    and consequently
    \begin{align*}
        \d_1\rho(\d_2,\d_3) + \d_2\rho(\d_3,\d_1)+&\d_3\rho(\d_1,\d_2)
        =2\d_1\rho(\d_2,\d_3) + 2\d_2\rho(\d_3,\d_1)+2\d_3\rho(\d_1,\d_2)\\
        &-\rho([\d_1,\d_2],\d_3)-\rho([\d_2,\d_3],\d_1)-\rho([\d_3,\d_1],\d_2)
    \end{align*}
    which is equivalent to $d\rho=0$.
\end{proof}

\noindent
Recall that in the classical geometric setting, where the hermitian form is induced by a Riemannian metric (cf. Section~\ref{sec:gen.metric.symmetry}), $\rho=0$ and the hermitian calculus is always symmetric and satisfies the necessary condition in Proposition~\ref{prop:LC.necessary.cond} for the existence of a Levi-Civita connection. However, even in the classical case, the condition becomes nontrivial if we allow for an arbitrary hermitian form on the complexified cotangent bundle. Moreover, we note that a similar result holds for "weak quantum Levi-Civita connections" \cite[Sec 8.1]{bm:Quantum.Riemannian.geometry} as a consequence of the vanishing of torsion and cotorsion. 

\subsection{A sufficient condition for projective modules}\label{sec:suff.cond.proj.mods}

\noindent
As shown in Proposition~\ref{prop:LC.necessary.cond}, a necessary condition for the existence of a Levi-Civita condition is that the hermitian calculus is weakly symmetric. Furthermore, a necessary and sufficient condition is given in Theorem~\ref{thm:LC.inverse.image}, which may be cumbersome to check in practice. In this section, we will present sufficient conditions for the existence of Levi-Civita connections on finitely generated projective modules. As a consequence, we show that being weakly symmetric is sufficient in the case of free modules with a basis dual to that of the Lie algebra $\g$.  

Let $(\A,\g,h)$ be a finitely generated projective hermitian calculus, and let $\{\theta^i\}_{i=1}^N$ be a set of generators of $\Omegaoneg$. Moreover, let $\{\d_a\}_{a=1}^n$ be a hermitian basis of $\g$ and write $\theta^i_a=\theta^i(\d_a)$ as well as $\theta_{ia}=\theta_i(\d_a)$. Given $\nabla\in\Cg(\Omegaoneg)$ there exists $\Gamma^i_{aj}\in\A$ such that
\begin{align}\label{eq:connection.in.comp}
    \nabla_{\d_a}\theta^i = \Gamma_{aj}^i\theta^j
\end{align}
for $a=1,\ldots,n$ and $i=1,\ldots,N$. Demanding that $\nabla$ is compatible with $h$ amounts to
\begin{align*}
    &\d_a h(\theta^i,\theta^j) = h(\nabla_{\d_a}\theta^i,\theta^j)+h(\theta^i,\nabla_{\d_a}\theta^j)
    \equivalent\\
    &\d_a h^{ij} = \Gamma^i_{ak}h^{kj} + h^{ik}(\Gamma^j_{ak})^\ast\equivalent
    \d_a h^{ij} = \Gamma^i_{ak}h^{kj} + (\Gamma^j_{ak}h^{ki})^\ast.
\end{align*}
Thus, $\nabla$ is compatible with $h$ if there exists $U^{ij}_a\in\A$ such that $(U^{ij}_a)^\ast=U^{ji}_a$ and
\begin{align*}
    \Gamma_{ak}^i h^{kj} = \tfrac{1}{2}\d_a h^{ij} + iU^{ij}_a,
\end{align*}
implying that
\begin{align}\label{eq:metric.conn.in.comp}
    \nabla_{\d_a}\theta^i = \Gamma^i_{ak}\theta^k = \Gamma^i_{ak}h^{kj}h_{jl}\theta^l
    =\paraa{\thalf \d_a h^{ij}+iU^{ij}_a}\theta_j.
\end{align}
Furthermore, demanding $\nabla$ to be torsion free gives
\begin{align*}
    &\paraa{\nabla_{\d_a}\theta^i}(\d_b)-\paraa{\nabla_{\d_b}\theta^i}(\d_a)=d\theta^i(\d_a,\d_b)\equivalent\\
    &\paraa{\thalf\d_a h^{ij}+iU^{ij}_a}\theta_{jb}
    -\paraa{\thalf\d_b h^{ij}+iU^{ij}_b}\theta_{ja}=d\theta^i(\d_a,\d_b)
\end{align*}
which is equivalent to
\begin{align}\label{eq:lc.cond.V}
    \thalf(\d_a h^{ij})\theta_{jb}-\thalf(\d_b h^{ij})\theta_{ja}-d\theta^i(\d_a,\d_b)
    =iU^{ij}_b\theta_{ja}-iU^{ij}_a\theta_{jb}.
\end{align}
In the case $\{\theta^i\}_{i=1}^N$ is a basis for $\Omegaoneg$, it is sufficient to solve the above equation for $U^{ij}_a$ such that $(U^{ij}_a)^\ast=U^{ji}_a$ and then define a Levi-Civita connection using \eqref{eq:metric.conn.in.comp}. However, for projective modules, one has to be more careful to guarantee the existence of the connection. In any case, let us now describe how one can solve \eqref{eq:lc.cond.V} for $U^{ij}_a$ in certain situations.

Multiplying \eqref{eq:lc.cond.V} by $i\theta_{ic}^\ast$ from the left gives
\begin{align}\label{eq:theta.V.theta}
    \tfrac{i}{2}\theta_{ic}^\ast(\d_a h^{ij})\theta_{jb}-\tfrac{i}{2}\theta_{ic}^\ast(\d_b h^{ij})\theta_{ja}
    -i\theta_{ic}^\ast d\theta^i(\d_a,\d_b)
    =\theta_{ic}^\ast U^{ij}_a\theta_{jb}-\theta_{ic}^\ast U^{ij}_b\theta_{ja}.
\end{align}
Let $R_a\in\Mat_n(\A)$ denote the $n\times n$ matrix with entries in $\A$ defined by
\begin{align*}
    (R_a)_{bc} = \theta_{ib}^\ast U^{ij}_a\theta_{jc}
\end{align*}
satisfying 
\begin{align*}
    (R_a)_{bc}^\ast = \paraa{\theta_{ib}^\ast U^{ij}_a\theta_{jc}}^\ast
    =\theta_{jc}U^{ji}_a\theta_{ib} = (R_a)_{cb},
\end{align*}
i.e. $R_a$ is a hermitian matrix.
In this notation, one writes \eqref{eq:theta.V.theta} as
\begin{align}\label{eq:R.R.F}
    (R_a)_{cb} - (R_b)_{ca} = F_{cab}
\end{align}
with
\begin{align}\label{eq:Fcab}
    F_{cab} = 
    \tfrac{i}{2}\theta_{ic}^\ast(\d_a h^{ij})\theta_{jb}
    -\tfrac{i}{2}\theta_{ic}^\ast(\d_b h^{ij})\theta_{ja}
    -i\theta_{ic}^\ast d\theta^i(\d_a,\d_b).
\end{align}
satisfying $F_{cab} = -F_{cba}$. The next result tells us when it is possible to solve \eqref{eq:R.R.F}.

\begin{proposition}\label{prop:R.solution}
  Let $F_{abc}\in\A$ for $a,b,c\in\{1,\ldots,n\}$ such that
  $F_{abc}=-F_{acb}$. Then there exist hermitian
  matrices $R_1,\ldots,R_n\in\Mat_n(\A)$ such that
  \begin{align}
    (R_{a})_{cb}-(R_{b})_{ca}=F_{cab}\label{eq:Req}
  \end{align}
  for $a,b,c\in\{1,\ldots,n\}$, if and only if
  \begin{align}
    F_{abc} + F_{bca}+ F_{cab} + (F_{abc} + F_{bca}+ F_{cab})^\ast  = 0\label{eq:F.nec.cond}
  \end{align}
  for $a,b,c\in\{1,\ldots,n\}$.
\end{proposition}

\begin{remark}\label{rem:F.only.p.dist}
  Note that since $F_{abc}=-F_{acb}$, condition \eqref{eq:F.nec.cond} is
  trivially satisfied if two of the indices coincide. Hence,
  \eqref{eq:F.nec.cond} only gives a nontrivial condition for a choice of
  pairwise distinct $a,b,c\in\{1,\ldots,n\}$.
\end{remark}

\begin{proof}
  Let us start by fixing a triple of indices
  $a,b,c\in\{1,\ldots,n\}$. By cyclically permuting $a,b,c$ in
  eq~\eqref{eq:Req} one obtains the following system of equations
  \begin{align}
    &(R_{a})_{cb} - (R_{b})_{ca} = F_{cab}\label{eq:Req1}\\
    &(R_{b})_{ac} - (R_{c})_{ab} = F_{abc}\label{eq:Req2}\\
    &(R_{c})_{ba} - (R_{a})_{bc} = F_{bca}\label{eq:Req3}
  \end{align}
  Adding these three equations together, along with their $\ast$-conjugates gives
  \begin{align*}
    F_{abc} + F_{bca}+ F_{cab} + (F_{abc} + F_{bca}+ F_{cab})^\ast  = 0,
  \end{align*}
  by using that $(R_a)_{bc}^\ast=(R_a)_{cb}$, which proves the
  necessity of \eqref{eq:F.nec.cond} for any choice of
  $a,b,c\in\{1,\ldots,n\}$. Next, let us show that \eqref{eq:F.nec.cond}
  is also a sufficient condition for solving \eqref{eq:Req}. We start
  by considering the case $n\geq 3$.

  For $a=b$, equation \eqref{eq:Req} is identically satisfied since
  $F_{caa}=0$. Moreover, for $a=c$ (or equivalently $b=c$), equation \eqref{eq:Req} decouples into equations
  \begin{align}
    (R_{a})_{ab}-(R_b)_{aa} = F_{aab}
  \end{align}
  for each choice of $a\neq b\in\{1,\ldots,n\}$.  Consequently, for
  $a=c$ (or $b=c$) one solves \eqref{eq:Req} by setting
  \begin{align}
    (R_{a})_{ab} = (R_b)_{aa}+F_{aab}\label{Raab.solution}
  \end{align}
  for $a\neq b\in\{1,\ldots,n\}$; i.e. the diagonal elements of the
  matrices $\{R_a\}_{a=1}^n$ will be arbitrary parameters in the
  general solution. It remains to solve \eqref{eq:Req} for pairwise
  distinct $a,b,c\in\{1,\ldots,n\}$.

  To this end, fix a triple $a,b,c\in\{1,\ldots,n\}$ of pairwise
  distinct integers, and using $(R_{a})_{bc}^\ast=(R_a)_{cb}$ one
  writes equations \eqref{eq:Req1}--~\eqref{eq:Req3} as
  \begin{align}
    &(R_{a})_{bc}^\ast - (R_{b})_{ca} = F_{cab}\label{eq:Req1.s}\\
    &(R_{b})_{ca}^\ast - (R_{c})_{ab} = F_{abc}\label{eq:Req2.s}\\
    &(R_{c})_{ab}^\ast - (R_{a})_{bc} = F_{bca}\label{eq:Req3.s}.
  \end{align}
  In this case, one notes that \eqref{eq:Req} decouples into such sets
  of three equations for every choice of pairwise distinct
  $a,b,c\in\{1,\ldots,n\}$. From \eqref{eq:Req1.s} one obtains
  \begin{align*}
    (R_b)_{ca} = (R_{a})_{bc}^\ast-F_{cab}\implies
    (R_b)_{ca}^\ast = (R_{a})_{bc}-F_{cab}^\ast
  \end{align*}
  which, inserted in \eqref{eq:Req2.s}, gives
  \begin{align*}
    (R_a)_{bc} = (R_c)_{ab}+F_{abc}+F_{cab}^\ast.
  \end{align*}
  Finally, inserting the above equation in \eqref{eq:Req3.s} gives
  \begin{align}
    (R_c)_{ab}^\ast-(R_c)_{ab} = F_{abc} + F_{bca} + F_{cab}^\ast.\label{eq:Rcab.final}
  \end{align}
  The left hand side of the above equation is clearly antihermitian,
  and the assumption \eqref{eq:F.nec.cond} implies that the right hand
  side is also antihermitian. Hence, \eqref{eq:Rcab.final} has the
  general solution
  \begin{align*}
    (R_c)_{ab} = -\tfrac{1}{2}\paraa{F_{abc}+F_{bca}+F_{cab}^\ast} + H_{cab}
  \end{align*}
  for arbitrary $H_{cab}=H_{cab}^\ast\in\A$, which implies that the
  general solution of \eqref{eq:Req1.s}--~\eqref{eq:Req3.s} is given
  by
  \begin{align}
    &(R_{a})_{bc} = \tfrac{1}{2}\paraa{F_{abc}-F_{bca}+F_{cab}^\ast} + H_{cab}\\
    &(R_{b})_{ca} = \tfrac{1}{2}\paraa{F_{abc}^\ast-F_{bca}^\ast-F_{cab}} + H_{cab}\\
    &(R_{c})_{ab} = -\tfrac{1}{2}\paraa{F_{abc}+F_{bca}+F_{cab}^\ast} + H_{cab}
  \end{align}
  for arbitrary $H_{cab}=H_{cab}^\ast\in\A$, which together with
  \eqref{Raab.solution} gives the general solution of
  \eqref{eq:Req}. This shows that \eqref{eq:F.nec.cond} is a sufficient
  condition for solving \eqref{eq:Req} for $n\geq 3$.

  Finally, let us consider the cases $n=1,2$. As noted in
  Remark~\ref{rem:F.only.p.dist}, condition \eqref{eq:F.nec.cond} is
  trivially satisfied unless $a,b,c$ are pairwise distinct, which is
  not possible for $n\in\{1,2\}$.  For $n=1$ is is clear that
  \eqref{eq:Req} is satisfied for any choice of $(R_1)_{11}\in\A$. For
  $n=2$, equation \eqref{eq:Req} reduces to
  \begin{align*}
    &(R_1)_{12}-(R_2)_{11} = F_{112}\qquad\quad(R_2)_{21}-(R_{1})_{22} = F_{221}
  \end{align*}
  which is solved by setting
  \begin{align*}
    &(R_1)_{12} = (R_2)_{11} + F_{112}\qquad\quad (R_2)_{21} = (R_{1})_{22} + F_{221}
  \end{align*}
  for arbitrary $(R_2)_{11},(R_1)_{22}\in\A$. 
\end{proof}

\noindent
One finds that $F_{abc}$, as defined in \eqref{eq:Fcab}, satisfies $F_{abc}=-F_{acb}$ as well as
\begin{align*}
    F_{abc} + &F_{bca} + F_{cab}
    +F_{abc}^\ast + F_{bca}^\ast + F_{cab}^\ast \\
    &=\paraa{i\theta_i^\ast(d h^{ik})\theta_k
    -i\theta_i^\ast d\theta^i+i(d\theta^i)^\ast \theta_i}(\d_a,\d_b,\d_c)\\
    &=id\rho(\d_a,\d_b,\d_c),
\end{align*}
by recalling from Lemma~\ref{lemma:rho.in.generators} the expression for $d\rho$ in terms of generators. Thus, if the calculus is weakly symmetric then one can use Proposition~\ref{prop:R.solution} to find hermitian matrices $\{R_1,\ldots,R_n\}$ solving \eqref{eq:Req}. Let us now use these results to derive a sufficient condition for weakly symmetric finitely generated projective hermitian calculi.

\begin{proposition}\label{prop:LC.sufficient.cond}
    Let $(\A,\g,h)$ be a weakly symmetric finitely generated projective hermitian calculus, let $\{\theta^i\}_{i=1}^N$ be generators of $\Omegaoneg$ and let $\{\d_a\}_{a=1}^n$ be a hermitian basis of $\g$. If 
    \begin{align*}
        \d(h^{ij}h_{jk})\theta^k=0,
    \end{align*}
    for $i=1,\ldots,N$ and $\d\in\g$, and
    there exist $\varphi^{ai}\in\A$ such that
    \begin{align*}
        &\varphi^{ai}\theta_{ib}=\delta^a_b\mid\\
        &(\theta_{ia}\varphi^{aj})^\ast\theta^i=\theta^j
    \end{align*}
    for $a,b=1,\ldots,n$ and $j=1,\ldots,N$ (where $\theta_{ia}=\theta_i(\d_a)$), then there exists a Levi-Civita connection on $\Omegaoneg$.
\end{proposition}

\begin{proof}
    Let us first note that if $(\theta_{ia}\varphi^{aj})^\ast\theta^i=\theta^j$ then
    \begin{align*}
        (\theta_{ia}\varphi^{aj})^\ast 
        = \paraa{h_{ik}h^{kl}\theta_{la}\varphi^{aj}}^\ast
        = \paraa{\theta_{la}\varphi^{aj}}^\ast h^{lk}h_{ki} 
        =h^{jk}h_{ki}
    \end{align*}
    which is equivalent to
    \begin{equation}\label{eq:theta.varphi.eq.hh}
        \theta_{ia}\varphi^{aj} = h_{ik}h^{kj}. 
    \end{equation}
    Now, set
    \begin{align*}
        F^i_{ab} = 
        \tfrac{i}{2}(\d_a h^{ij})\theta_{jb}
        -\tfrac{i}{2}(\d_b h^{ij})\theta_{ja}
        -id\theta^i(\d_a,\d_b)
    \end{align*}
    and $F_{cab}=\theta^\ast_{ic}F^i_{ab}$, clearly satisfying $F_{cab}=-F_{cba}$. Since $d\rho=0$, we let $\{R_1,\ldots,R_n\}\subseteq\Mat_n(\A)$ be hermitian matrices solving \eqref{eq:Req} in Proposition~\ref{prop:R.solution}. Next, set
    \begin{align*}
        U^{ik}_a = (\varphi^{cl}h_{lj}h^{ji})^\ast (R_a)_{cb}\varphi^{bm}h_{mn}h^{nk},
    \end{align*}
    from which it follows that $(U_a^{ik})^\ast=U^{ki}_a$ as well as
    \begin{align*}
        h^{ij}h_{jk}U^{kl}_a = U^{il}_a\qand
        U^{ij}_a h_{jk}h^{kl} = U^{il}_a.
    \end{align*}
    Furthermore, using that (cf. \eqref{eq:theta.varphi.eq.hh})
    \begin{align*}
        (\varphi^{ck}h_{kj}h^{ji})^\ast\theta_{lc}^\ast
        =(\theta_{lc}\varphi^{ck}h_{kj}h^{ji})^\ast
        =(h_{lm}h^{mk}h_{kj}h^{ji})^\ast
        =(h_{lm}h^{mi})^\ast = h^{im}h_{ml}
    \end{align*}
    one finds that
    \begin{align*}
        U^{ik}_a\theta_{kb}-U^{ik}_b\theta_{ka}
        &=(\varphi^{cl}h_{lj}h^{ji})^\ast\paraa{(R_a)_{cb}-(R_b)_{ca}}
        =(\varphi^{cl}h_{lj}h^{ji})^\ast F_{cab}\\
        &=h^{ik}h_{kl}\paraa{\tfrac{i}{2}(\d_a h^{lj})\theta_{jb}
        -\tfrac{i}{2}(\d_b h^{lj})\theta_{ja}
        -id\theta^l(\d_a,\d_b)}.
    \end{align*}
    Since $(\d_a h^{ik}h_{kl})\theta^l=0$ one obtains
    \begin{equation}
        \begin{split}
        U^{ik}_a\theta_{kb}-U^{ik}_b\theta_{ka}
        &=\tfrac{i}{2}(\d_a h^{ik}h_{kl}h^{lj})\theta_{jb}
        -\tfrac{i}{2}(\d_b h^{ik}h_{kl}h^{lj})\theta_{ja}
        -id(h^{ik}h_{kl}\theta^l)(\d_a,\d_b)\\
        &=\tfrac{i}{2}(\d_a h^{ij})\theta_{jb}
        -\tfrac{i}{2}(\d_b h^{ij})\theta_{ja}
        -id\theta^i(\d_a,\d_b) = F^i_{ab}.\label{eq:UU.Fiab}    
        \end{split}
    \end{equation}
    Let us now construct a connection as in \eqref{eq:metric.conn.in.comp}, i.e. for $f_i\theta^i\in\Omegaoneg$ we set
    \begin{align}\label{eq:def.LC.proj.mod}
        \nabla_{\d_a}(f_i\theta^i)
        &=f_i\paraa{\tfrac{1}{2}(\d_a h^{ik})\theta_k+iU^{ik}_a\theta_k} + (\d_a f_i)\theta^i.
    \end{align}
    For this to be well-defined on the projective module $\Omegaoneg$, one needs to check that $f_i\theta^i=0$ implies that $\nabla_{\d_a}(f_i\theta^i)=0$. Thus, assuming $f_i\theta^i=0$, one computes
    \begin{align*}
        \nabla_{\d_a}(f_i\theta^i)
        &=f_i\paraa{\tfrac{1}{2}(\d_a h^{ik})\theta_k+iU^{ik}_a\theta_k} + (\d_a f_i)\theta^i\\
        &=\tfrac{1}{2}(\d_a f_ih^{ik})\theta_k-\tfrac{1}{2}(\d_a f_i)\theta^i+if_iU^{ik}_a+(\d_af_i)\theta^i
        =\tfrac{1}{2}(\d_a f_i)\theta^i
    \end{align*}
    since $f_ih^{ik}=h(f_i\theta^i,\theta^k)=0$ and $f_iU^{ik}_a=f_ih^{ij}h_{jk}U^{kl}_a=0$. Moreover, using that $(\d_a h^{ij}h_{jk})\theta^k=0$, one finds that
    \begin{align*}
        \nabla_{\d_a}(f_i\theta^i)
        &=\thalf(\d_a f_i)\theta^i
        = \thalf(\d_a f_i)h^{ij}h_{jk}\theta^k 
        = \thalf\d_{a}(f_ih^{ij}h_{jk})\theta^k = 0
    \end{align*}
    since $f_ih^{ij}=0$. We conclude that \eqref{eq:def.LC.proj.mod} defines a connection on $\Omegaoneg$. The connection is compatible with $h$; namely,
    \begin{align*}
        \d_ah^{ij}-&h(\nabla_{\d_a}\theta^i,\theta^j)-h(\theta^i,\nabla_{\d_a}\theta^j)\\
        &= \d_a h^{ij}-\paraa{\thalf(\d_a h^{ik})+iU^{ik}_a}h_{kl}h^{lj}
        -h^{il}h_{lk}\paraa{\thalf(\d_a h^{jk})+iU^{jk}_a}^\ast\\
        &= \d_a h^{ij} -\thalf(\d_a h^{ik})h_{kl}h^{lj}-\thalf h^{il}h_{lk}(\d_a h^{kj})
    \end{align*}
    since $iU^{ik}_a h_{kl}h^{lj}+h^{il}h_{lk}(iU^{jk}_a)^\ast=iU^{ij}_a-iU^{ij}_a=0$. Again, using that 
    $(\d_a h^{ij}h_{jk})\theta^k=0$ one finds that
    \begin{align*}
        \d_ah^{ij}-&h(\nabla_{\d_a}\theta^i,\theta^j)-h(\theta^i,\nabla_{\d_a}\theta^j)\\
        &= \d_a h^{ij} -\thalf(\d_a h^{ik})h_{kl}h^{lj}-\thalf h^{il}h_{lk}(\d_a h^{kj})\\
        &= \d_a h^{ij} -\thalf\paraa{h^{jl}h_{lk}(\d_ah ^{ki})}^\ast -\thalf h^{il}h_{lk}(\d_a h^{kj})\\
        &= \d_a h^{ij} -\thalf\paraa{\d_a h^{jl}h_{lk}h ^{ki}}^\ast -\thalf \d_a h^{il}h_{lk}h^{kj}\\
        &= \d_a h^{ij} - \thalf\d_a h^{ij} - \thalf\d_a h^{ij} =0,
    \end{align*}
    showing that $\nabla$ is compatible with $h$. Finally, $\nabla$ is torsion free since
    \begin{align*}
        (\nabla_{\d_a}\theta^i)(\d_b)&-(\nabla_{\d_a}\theta^i)(\d_b)-d\theta^i(\d_a,\d_b)\\
        &=\thalf (\d_a h^{ik})\theta_{kb}+iU^{ik}_a\theta_{kb}
        -\thalf (\d_b h^{ik})\theta_{ka}-iU^{ik}_b\theta_{ka}-d\theta^i(\d_a,\d_b)\\
        &=iU^{ik}_a\theta_{kb}-iU^{ik}_b\theta_{ka}-iF^i_{ab} = 0
    \end{align*}
    as shown above. Hence, $\nabla$ is a torsion free connection on $\Omegaoneg$ compatible with $h$.
\end{proof}

\noindent 
Let us apply the result above to the simplest possible situation of a free module with a basis that is dual to a hermitian basis of the Lie algebra $\g$. 

\begin{corollary}\label{cor:LC.free.iff.wsym}
    Assume that $(\A,\g,h)$ is a finitely generated free hermitian calculus such that $\dim(\g)=\operatorname{rk}(\Omegaoneg)=n$. Moreover, assume that there exists a basis $\{\theta^i\}_{i=1}^n$ of $\Omegaoneg$ and a basis 
    $\{\d_a\}_{a=1}^n$ of $\g$ such that $\theta^i(\d_a)=\delta^i_a\mid$ for $i,a=1,\ldots,n$. Then there exists a Levi-Civita connection on $\Omegaoneg$ if and only if $(\A,\g,h)$ is weakly symmetric.
\end{corollary}

\begin{proof}
    It follows from Proposition~\ref{prop:LC.necessary.cond} that if there exists a Levi-Civita connection on $\Omegaoneg$ then $(\A,\g,h)$ is weakly symmetric. Conversely, assume that $(\A,\g,h)$ is weakly symmetric and assume that there exists a basis 
    $\{\theta^i\}_{i=1}^n$ such that $\theta^i(\d_a)=\delta^i_a\mid$. This implies that 
    $\theta_{ia}=h_{ik}\theta^k(\d_a)=h_{ia}$. Choosing $\varphi^{ai}=h^{ai}$ gives $\varphi^{ai}\theta_{ib}=h^{ai}h_{ib}=\delta^a_b\mid$ as well as 
    \begin{align*}
        (\theta_{ia}\varphi^{aj})^\ast\theta^i
        &=(h_{ia}h^{aj})^\ast\theta^i = (\delta_i^j\mid)^\ast\theta^i = \theta^j.
    \end{align*}
    Furthermore, $\d_a(h^{ij}h_{jk})\theta^k=\d_a(\delta^i_k\mid)\theta^k=0$. It follows from Proposition~\ref{prop:LC.sufficient.cond} that there exists a Levi-Civita connection on $\Omegaoneg$.
\end{proof}

\begin{remark}
    Note that if $(\A,\g,h)$ is a finitely generated free hermitian calculus such that $\dim(\g)=\operatorname{rk}(\Omegaoneg)=n$, with a basis satisfying $\theta^i(\d_a)=\delta^i_a\mid$ then $\Omegaoneg=\Omegaonebg$, since for $f\in\Omegaonebg$ one finds that
    \begin{align*}
        f(\d_k)\theta^k(\d_i) = f(\d_k)\delta^k_i = f(\d_i) 
    \end{align*}
    for $i=1,\ldots,n$, implying that $f\in\Omegaoneg$.
\end{remark}

\section{Levi-Civita connections on free modules of rank 3}\label{sec:lc.free.modules}

\noindent
As an illustration of the results in this paper, let us explicitly compute Levi-Civita connections for weakly symmetric hermitian calculi where $\Omegaoneg$ is a free module such that $\operatorname{rk}(\Omegaoneg)=\dim(\g)=3$. This is the lowest dimension for which weak symmetry is nontrivial. 

Let $(\A,\g,h)$ be a finitely generated free hermitian calculus such that $\dim(\g)=\operatorname{rk}(\Omegaoneg)=n$, and assume that there exists a basis $\{\theta^i\}_{i=1}^n$ for $\Omegaoneg$ such that $\theta^i(\d_a)=\delta^i_a\mid$ for $i,a=1,\ldots,n$, where $\{\d_a\}_{a=1}^n$ is a hermitian basis for $\g$. For presentational simplicity, we assume $\g$ to be abelian, as nonabelian Lie algebras do not pose any additional difficulties (and make formulas more lengthy). By the duality of the basis of $\Omegaoneg$ and the basis of $\mathfrak{g}$, the basis $\{\theta^i\}_{i=1}^n$ satisfies
\[
(\theta^j)^\ast=\theta^j,\quad[\theta^j,f]=\theta^jf-f\theta^j=0,\quad\theta^j\theta^i=-\theta^i\theta^j\text{ if }i\neq j,\quad d\theta^j=0
\]
where $f\in\A$ and $a,i,j=1,\dots,n$.

Corollary \ref{cor:LC.free.iff.wsym} states that, in the current situation, there exists a Levi-Civita connection if and only if $(\A,\g,h)$ is weakly symmetric, which may be constructed as
\begin{align}\label{eq:LC.def.fin.gen.example}
    \nabla_{\d_a}\theta^i=\paraa{\thalf \d_a h^{ij}+iU^{ij}_a}\theta_j,
\end{align}
for an appropriate choice of $U^{ij}_a\in\A$ such that $(U^{ij}_a)^\ast=U^{ji}_a$.

Recall that the free hermitian calculus $(\A,\g,h)$ is weakly symmetric precisely when 
\[
d\rho=0\equivalent d(\theta_i^\ast\theta^i)=0\equivalent F_{abc} + F_{bca} + F_{cab}
    +F_{abc}^\ast + F_{bca}^\ast + F_{cab}^\ast=0,
\]
where 
\[
F_{cab}= 
    \tfrac{i}{2}\theta_{ic}^\ast(\d_a h^{ij})\theta_{jb}
    -\tfrac{i}{2}\theta_{ic}^\ast(\d_b h^{ij})\theta_{ja}.
\]
for $a,b,c\in\{1,\dots,n\}$.
Next, we note that
\[
\theta_{ib}=\theta_i(\d_b)=h_{ij}\theta^j(\d_b)=h_{ij}\delta^j_b=h_{ib},
\]
giving $\theta_{ib}^\ast=h_{ib}^\ast=h_{bi}$, implying that $(R_a)_{bc}$ and $F_{cab}$ simplify to 
\[
(R_a)_{bc}=h_{bi}U^{ij}_ah_{jc},\quad F_{cab}=-\frac{i}{2}\d_ah_{cb}+\frac{i}{2}\d_bh_{ca},
\]
and $d\rho=0$ becomes 
\begin{align}\label{eq:wsym.condition.nxn}
    \d_a(h_{bc}-h_{bc}^\ast)+\d_b(h_{ca}-h_{ca}^\ast)+\d_{c}(h_{ab}-h_{ab}^\ast)=0
\end{align}
for $a,b,c\in\{1,\dots,n\}$. 

Now consider the case $n=3$. To obtain an expression for the Levi-Civita connection, one needs to find $U^{ij}_a$ given by the equations
\begin{equation}\label{eq:RRF}
    (R_a)_{cb}-(R_b)_{ca}=F_{cab}    
\end{equation}
for $a,b,c\in\{1,2,3\}$ where $(R_a)_{bc}=h_{bi}U^{ij}_ah_{jc}$. As in the proof of Proposition~\ref{prop:R.solution}, these equations are solved by
\[
(R_{a})_{ab} = (R_b)_{aa}+F_{aab}
\]
for $a=c\neq b\in\{1,2,3\}$, and 
\begin{align*}
    &(R_{a})_{bc} = \tfrac{1}{2}\paraa{F_{abc}-F_{bca}+F_{cab}^\ast} + H_{cab}\\
    &(R_{b})_{ca} = \tfrac{1}{2}\paraa{F_{abc}^\ast-F_{bca}^\ast-F_{cab}} + H_{cab}\\
    &(R_{c})_{ab} = -\tfrac{1}{2}\paraa{F_{abc}+F_{bca}+F_{cab}^\ast} + H_{cab}
  \end{align*}
when $a,b,c\in\{1,2,3\}$ are pairwise distinct, for arbitrary hermitian $H_{cab}\in\A$. Letting $R_a$ be the matrix with entries $(R_a)_{bc}$, it follows that
\begin{align*}
    R_1&=
    \begin{pmatrix}
        X_{11}&X_{21}-\frac{i}{2}\d_1h_{12}+\frac{i}{2}\d_2h_{11}&X_{31}-\frac{i}{2}\d_1h_{13}+\frac{i}{2}\d_3h_{11}\\
        X_{21}+\frac{i}{2}\d_1h_{21}-\frac{i}{2}\d_2h_{11}&X_{12}&-\frac{i}{2}\d_2h_{31}+\frac{i}{2}\d_3h_{21}+H_{123}\\
        X_{31}+\frac{i}{2}\d_1h_{31}-\frac{i}{2}\d_3h_{11}&\frac{i}{2}\d_2h_{13}-\frac{i}{2}\d_3h_{12}+H_{123}&X_{13}
    \end{pmatrix}\\
    R_2&=
    \begin{pmatrix}
        X_{21}&X_{12}+\frac{i}{2}\d_2h_{12}-\frac{i}{2}\d_1h_{22}&\frac{i}{2}\d_3h_{12}-\frac{i}{2}\d_1h_{32}+H_{123}\\
        X_{12}-\frac{i}{2}\d_2h_{21}+\frac{i}{2}\d_1h_{22}&X_{22}&X_{32}-\frac{i}{2}\d_2h_{23}+\frac{i}{2}\d_3h_{22}\\
        -\frac{i}{2}\d_3h_{21}+\frac{i}{2}\d_1h_{23}+H_{123}&X_{32}+\frac{i}{2}\d_2h_{32}-\frac{i}{2}\d_3h_{22}&X_{23}
    \end{pmatrix}\\
    R_3&=
    \begin{pmatrix}
        X_{31}&-\frac{i}{2}\d_1h_{32}+\frac{i}{2}\d_2h_{13}+H_{123}&X_{13}+\frac{i}{2}\d_3h_{13}-\frac{i}{2}\d_1h_{33}\\
        \frac{i}{2}\d_1h_{23}-\frac{i}{2}\d_2h_{31}+H_{123}&X_{32}&X_{23}+\frac{i}{2}\d_3h_{23}-\frac{i}{2}\d_2h_{33}\\
        X_{13}-\frac{i}{2}\d_3h_{31}+\frac{i}{2}\d_1h_{33}&X_{23}-\frac{i}{2}\d_3h_{32}+\frac{i}{2}\d_2h_{33}&X_{33}
    \end{pmatrix}
\end{align*}
solve \eqref{eq:RRF} for arbitrary hermitian $X_{ab},H_{123}\in\A$. 

To write down an explicit example of a Levi-Civita connection, let us choose a particular metric of the form
\[
h=(h^{ij})_{i,j=1}^3=
\begin{pmatrix}
    \mathds{1}&0&0\\
    0&0&h_0\\
    0&h_0^\ast&0
\end{pmatrix},\quad h^{-1}=(h_{ij})_{i,j=1}^3=
\begin{pmatrix}
    \mathds{1}&0&0\\
    0&0&(h_0^{-1})^\ast\\
    0&h_0^{-1}&0
\end{pmatrix}
\]
where $h_0\in\A$ is invertible.  Note that we have chosen a metric with off-diagonal parts since, for a purely diagonal metric, condition \eqref{eq:wsym.condition.nxn} is immediately satisfied.

Thus, for the above metric, the condition of weakly symmetric becomes 
\[
\d_1((h_0^{-1})^\ast)-\d_1(h_0^{-1})=0,
\]
giving a necessary condition for the existence of a Levi-Civita connection for a metric of the above form.

Assume now that $H_{123}=0$ and $X_{ab}=0$ for all $a,b,c\in\{1,2,3\}$ except $a=b=1$, keeping $X_{11}$ non-zero to demonstrate the non-uniqueness of the Levi-Civita connection. The matrices $\{R_a\}_{a=1}^3$ become
\begin{align*}
    R_1&=
    \begin{pmatrix}
        X_{11}&0&0\\
        0&0&0\\
        0&0&0
    \end{pmatrix}\qquad
    R_2=
    \begin{pmatrix}
        0&0&-\frac{i}{2}\d_1(h_0^{-1})\\
        0&0&-\frac{i}{2}\d_2((h_0^{-1})^\ast)\\
        \frac{i}{2}\d_1((h_0^{-1})^\ast)&\frac{i}{2}\d_2(h_0^{-1})&0
    \end{pmatrix}\\
    R_3&=
    \begin{pmatrix}
        0&-\frac{i}{2}\d_1(h_0^{-1})&0\\
        \frac{i}{2}\d_1((h_0^{-1})^\ast)&0&\frac{i}{2}\d_3((h_0^{-1})^\ast)\\
        0&-\frac{i}{2}\d_3(h_0^{-1})&0
    \end{pmatrix}.
\end{align*}
By construction $U_a=hR_ah$, where $U_a=(U_a^{ij})_{i,j=1}^3$, giving
\begin{align*}
    U_1&=
    \begin{pmatrix}
        X_{11}&0&0\\
        0&0&0\\
        0&0&0
    \end{pmatrix}\quad
    U_2=
    \begin{pmatrix}
        0&-\frac{i}{2}\d_1(h_0^{-1})h_0^\ast&0\\
        \frac{i}{2}h_0\d_1((h_0^{-1})^\ast)&0&\frac{i}{2}h_0\d_2(h_0^{-1})h_0\\
        0&-\frac{i}{2}h_0^\ast\d_2((h_0^{-1})^\ast)h_0^\ast&0
    \end{pmatrix}\\
    U_3&=
    \begin{pmatrix}
        0&0&-\frac{i}{2}\d_1(h_0^{-1})h_0\\
        0&0&-\frac{i}{2}h_0\d_3(h_0^{-1})h_0\\
        \frac{i}{2}h_0^\ast\d_1((h_0^{-1})^\ast)&\frac{i}{2}h_0^\ast\d_3((h_0^{-1})^\ast)h_0^\ast&0
    \end{pmatrix}.
\end{align*}
The Levi-Civita connection \eqref{eq:LC.def.fin.gen.example} is then given by
\begin{align*}
    \nabla_{\d_1}\theta^1&=iX_{11}\theta^1,\quad\nabla_{\d_1}\theta^2=-\frac{1}{2}h_0\d_1(h_0^{-1})\theta^2,\quad\nabla_{\d_1}\theta^3=-\frac{1}{2}h_0^\ast\d_1((h_0^{-1})^\ast)\theta^3,\\
    \nabla_{\d_2}\theta^1&=\frac{1}{2}\d_1(h_0^{-1})\theta^3,\quad\nabla_{\d_2}\theta^2=-\frac{1}{2}h_0\d_1((h_0^{-1})^\ast)\theta^1-h_0\d_2(h_0^{-1})\theta^2,\quad\nabla_{\d_2}\theta^3=0\\
    \nabla_{\d_3}\theta^1&=\frac{1}{2}\d_1(h_0^{-1})\theta^2,\quad\nabla_{\d_3}\theta^2=0,\quad\nabla_{\d_3}\theta^3=-\frac{1}{2}h_0^\ast\d_1(h_0^{-1})\theta^1-h_0^\ast\d_3(h_0^{-1})\theta^3.
\end{align*}
To get a better understanding for how the various conditions are satisfied, let us explicitly verify that $\nabla$ is torsion free and compatible with $h$. One checks that
\begin{align*}
    h(\nabla_{\d_1}\theta^2,\theta^3)&+h(\theta^2,\nabla_{\d_1}\theta^3)=-\frac{1}{2}h_0\d_1(h_0^{-1})h_0-\frac{1}{2}(h_0^\ast\d_1((h_0^{-1})^\ast)h_0^\ast)^\ast=\\
    &=\frac{1}{2}\d_1(h_0)h_0^{-1}h_0+\frac{1}{2}(\d_1(h_0^\ast)(h_0^{-1})^\ast h_0^\ast)^\ast=\frac{1}{2}\d_1h_0+\frac{1}{2}\d_1(h_0^\ast)^\ast=\d_1h^{23}
\end{align*}
which also implies
\[
h(\nabla_{\d_1}\theta^3,\theta^2)+h(\theta^3,\nabla_{\d_1}\theta^2)=(h(\nabla_{\d_1}\theta^2,\theta^3)+h(\theta^2,\nabla_{\d_1}\theta^3))^\ast=(\d_1h_0)^\ast=\d_1h^{32}.
\]
Analogous computations give
\begin{align*}
    h(\nabla_{\d_2}\theta^2,\theta^3)+h(\theta^2,\nabla_{\d_2}\theta^3)&=\d_2h^{23},\quad h(\nabla_{\d_2}\theta^3,\theta^2)+h(\theta^3,\nabla_{\d_2}\theta^2)=\d_2h^{32},\\
    h(\nabla_{\d_3}\theta^2,\theta^3)+h(\theta^2,\nabla_{\d_3}\theta^3)&=\d_3h^{23},\quad h(\nabla_{\d_3}\theta^3,\theta^2)+h(\theta^3,\nabla_{\d_3}\theta^2)=\d_3h^{32},
\end{align*}
and it is easy to check that for all other triples $a,b,c\in\{1,2,3\}$
\[
\d_ch^{ab}=0=h(\nabla_{\d_c}\theta^a,\theta^b)+h(\theta^a,\nabla_{\d_c}\theta^b),
\]
showing that $\nabla$ is indeed compatible with $h$. Let us now turn to the torsion free condition.

Since $d\theta^i=0$, $\nabla$ is torsion free if $(\nabla_{\d_a}\theta^i)(\d_b)=(\nabla_{\d_b}\theta^i)(\d_a)$. Clearly if $a=b$ this condition is satisfied. One checks that
\begin{align*}
    (\nabla_{\d_1}\theta^1)(\d_2)&=iX_{11}\theta^1(\d_2)=0=\frac{1}{2}\d_1(h_0^{-1})\theta^3(\d_1)=(\nabla_{\d_2}\theta^1)(\d_1),\\
    (\nabla_{\d_1}\theta^1)(\d_3)&=iX_{11}\theta^1(\d_3)=0=\frac{1}{2}\d_1(h_0^{-1})\theta^2(\d_1)=(\nabla_{\d_3}\theta^1)(\d_1),\\
    (\nabla_{\d_1}\theta^2)(\d_3)&=-\frac{1}{2}h_0\d_1(h_0^{-1})\theta^2(\d_3)=0=(\nabla_{\d_3}\theta^2)(\d_1),\\
    (\nabla_{\d_1}\theta^3)(\d_2)&=-\frac{1}{2}h_0^\ast\d_1((h_0^{-1})^\ast)\theta^3(\d_2)=0=(\nabla_{\d_2}\theta^3)(\d_1),
\end{align*}
as well as
\begin{align*}
    (\nabla_{\d_1}\theta^2)(\d_2)&=-\frac{1}{2}h_0\d_1(h_0^{-1})\theta^2(\d_2)=-\frac{1}{2}h_0\d_1((h_0^{-1})^\ast)\theta^1(\d_1)=(\nabla_{\d_2}\theta^2)(\d_1),\\
    (\nabla_{\d_1}\theta^3)(\d_3)&=-\frac{1}{2}h_0^\ast\d_1((h_0^{-1})^\ast)\theta^3(\d_3)=\\
    &=-\frac{1}{2}h_0^\ast\d_1(h_0^{-1})\theta^1(\d_1)-h_0^\ast\d_3(h_0^{-1})\theta^3(\d_1)=(\nabla_{\d_3}\theta^3)(\d_1),
\end{align*}
where one explicitly uses the condition of weak symmetry $\d_1(\hi_0)=\d_1((\hi_0)^\ast)$, showing that $\nabla$ is torsion free and, consequently, a Levi-Civita connection. 
The connection above is defined in terms of an arbitrary (hermitian) parameter $X_{11}$, explicitly demonstrating the non-uniqueness of Levi-Civita connections in this context.

Finally, we note that all of the above is applicable to noncommutative tori. To this end, let $T^n_\theta$ be the noncommutative $n$-dimensional torus, i.e.\ the $\ast$-algebra generated by unitary $\{U_i\}_{i=1}^n$ such that $U_iU_j=q_{ij}U_jU_i$, where $q_{ij}=e^{2\pi i\theta_{ij}}$ for $\theta_{ij}\in\reals$ such that $\theta_{ij}=-\theta_{ji}$. Note that $\theta_{ii}=0$ and $q_{ji}=q_{ij}^{-1}$. Assuming that $\theta_{ij}$ is irrational when $i\neq j$ it follows that $Z(T^n_\theta)\simeq\complex$. 

Consider the standard derivations $\d_aU_j=i\delta_{aj}U_j$ for $a=1,\dots,n$. These are hermitian and $[\d_a,\d_b]=0$ for all $a,b=1,\dots,n$, and we let $\g$ denote the Lie algebra generated by $\{\d_a\}_{a=1}^n$. Then $\Omegaoneg$ is generated by $dU_j$, and $\{dU_j\}_{j=1}^n$ is a basis of $\Omegaoneg$; namely,
\begin{align*}
    f^jdU_j=0&\implies f^jdU_j(\d_a)=0
    \implies f^j\d_a U_j=0
    \implies f^a=0
\end{align*}
for $a=1,\ldots,n$. Moreover, since
\[
\paraa{(dU_i)U_j}(\d_a)=(\d_aU_i)U_j=i\delta_i^aU_iU_j=i\delta^a_iq_{ij}U_jU_i=q_{ij}U_j\d_aU_i
=\paraa{q_{ij}U_jdU_i}(\d_a),
\]
the bimodule structure of $\Omega_\mathfrak{g}^1$ is
\[
(dU_i)U_j=q_{ij}U_jdU_i.
\]
Defining $\theta^j=-iU_j^{-1}dU_j$ it is easy to check that $\{\theta^j\}_{j=1}^n$ is also a basis for $\Omegaoneg$ and, furthermore, that
\[
\theta^j(\d_a)=-iU_j^{-1}dU_j(\d_a)=\delta^j_a\mathds{1}.
\]
Hence, $(T^n_\theta,\g)$ satisfies the prerequisites of Corollary~\ref{cor:LC.free.iff.wsym} and one can apply the results in this section to construct a Levi-Civita connection on $T^n_\theta$.

\section*{Acknowledgements}

\noindent
We would like to thank E. Darp\"o for discussions and for suggesting several improvements to the manuscript. 

\bibliographystyle{alpha}
\bibliography{references}

\end{document}